\documentclass{amsart}
\usepackage{amsmath,amsthm,amsfonts,amssymb,latexsym,mathrsfs,graphicx}

\usepackage{hyperref}

\usepackage{enumerate}
\usepackage[shortlabels]{enumitem}
\usepackage{color}

\usepackage{tikz}
\usetikzlibrary{decorations.pathreplacing}

\usepackage{ytableau}
\newtheorem{thm}{Theorem}[section]

\newtheorem{pro}[thm]{Proposition}

\newtheorem{rem}[thm]{Remark}
\newtheorem{defi}[thm]{Definition}
\newtheorem{lem}[thm]{Lemma}
\newtheorem{core}[thm]{Corollary}

\def \leq {\leqslant}
\def \geq {\geqslant}

\def\exp{{\sf exp}}
\def\Aut{{\sf Aut}}

\def\Cay{{\sf Cay}}

\let\oldproofname=\proofname
\renewcommand{\proofname}{\rm\bf{\oldproofname}}

\begin{document}

\title{Finite groups with quadratic splitting fields for all Cayley graphs}
	\author[M. Arezoomand]{Majid Arezoomand$^*$}
\author[A. Abdollahi]{Alireza Abdollahi}
\author[T. Feng]{Tao Feng}

	\address{Majid Arezoomand, Department of Mathematics, Faculty of Sciences, Shahid Rajaee Teacher Training University, Tehran 16785-163, Irann\newline
		and \newline
		School of Mathematics,\newline
		Institute for Research in Fundamental Science(IPM)
		P.O. Box:19395-5746, Tehran, Iran.}
	\email{ arezoomand@sru.ac.ir(Coressponding author)}
\address{Alireza Abdollahi, Department of Pure Mathematics, Faculty of Mathematics and Statistics,\newline
		 University of Isfahan, Isfahan 81746-73441, Iran}
\email{a.abdollahi@math.ui.ac.ir}

\address{Tao Feng, School of Mathematics and Statistics, Beijing Jiaotong University, Beijing 100044, PR China}
	\email{tfeng@bjtu.edu.cn}

\author[Z. Akhlaghi]{Zeinab Akhlaghi}
		\address{Zeinab Akhlaghi, Faculty of Math. and Computer Sci., \newline Amirkabir University of Technology (Tehran Polytechnic), 15914 Tehran, Iran.\newline
		and \newline
		School of Mathematics,\newline
		Institute for Research in Fundamental Science(IPM)
		P.O. Box:19395-5746, Tehran, Iran.}
	\email{z\_akhlaghi@aut.ac.ir}

\thanks{$^*$ Corresponding author\\
	Part of this work was carried out while the first author was visiting Beijing Jiaotong University. He expresses his sincere thanks to the 111 Project of China (Grant No. B16002) for financial support and to the School of Mathematics and Statistics, Beijing Jiaotong University, for their kind hospitality. The research of M. Arezoomand was in part supported by a grant from IPM (No. 1405050011). T. Feng is supported by the National Natural Science Foundation of China under Grant 12271023.  Also Z. Akhlaghi is  supported by a  grant from IPM (No. 1405320014).}

%	\title{Cayley $2$-integral groups}
%	\author{ 
%		Majid Arezoomand $^{a,d}$ \footnote{Email: arezoomand@sru.ac.ir (corresponding author)}, Alireza Abdollahi $^{c}$ \footnote{E-mail:
%		a.abdollahi@math.ui.ac.ir },  Tao Feng $^{e}$ \footnote{Email: tfeng@bjtu.edu.cn} and Zeinab Akhlaghi $^{b,d}$ \footnote{Email:z\_akhlaghi@aut.ac.ir} \\

%		{\small\em$^a$  }\\
%{\small\em$^b$ Faculty of Mathematics and Computer Science, Amirkabir University of Technology (Tehran Polytechnic), 15914 Tehran, Iran}\\
%	{\small\em$^c$ Department of Pure Mathematics, Faculty of Mathematics and Statistics,}\\ 
%		{\small\em University of Isfahan, Isfahan 81746-73441, Iran}\\
%		{\small\em$^d$School of Mathematics, Institute for Research in Fundamental Sciences (IPM),}\\
%		{\small\em Tehran 19395-5746, Iran}\\
		
%		{\small\em$^e$ School of Mathematics and Statistics, Beijing Jiaotong University, Beijing 100044, PR China }
%	}

%	\renewcommand{\Affilfont}{\small\it}
%	\renewcommand\Authands{ and }
	\date{}
	\maketitle
	
	\begin{abstract}
		For a graph $\Gamma$, the splitting field of $\Gamma$ is defined as the splitting field of the characteristic polynomial of $\Gamma$ over rationals. The algebraic degree of $\Gamma$ is defined by the extension degree of its splitting field over rationals. 
Let $k$ be a positive integer. We call a finite group $G$  \textit{Cayley $k$-integral} if, for every inverse-closed subset $S$ of $G$, the algebraic degree of the Cayley graph $\Cay(G,S)$ does not exceed $k$. We give a complete classification of all finite Cayley $2$-integral groups. It is shown that a finite abelian group is Cayley $2$-integral if and only if it is isomorphic to one of the following forms: $G \cong \mathbb{Z}_2^r \times \mathbb{Z}_5^s$, $\mathbb{Z}_2^r \times \mathbb{Z}_4^s \times \mathbb{Z}_8^t$, or $\mathbb{Z}_2^r \times \mathbb{Z}_3^s \times \mathbb{Z}_{12}^t$, where $r, s, t \geq 0$. Furthermore, we prove that the set of finite non-abelian Cayley $2$-integral groups consists of the infinite family $Q_8 \times \mathbb{Z}_2^n$, with $n \geq 0$, and $22$ specific groups.
	\end{abstract}

\section{Introduction and main result}

The interplay between algebraic structures and graph-theoretic properties has long been a fertile area of mathematical research. One of the most influential constructions linking these two disciplines is the concept of a Cayley graph, a graphical representation of finite groups. A \textit{Cayley graph} over a group \( G \) with respect to an inverse-closed subset \( S \) of \( G \), denoted by \( \Cay(G, S) \), is a graph with vertex set \( G \) such that $\{g,h\}$ forms an edge if and only if \( hg^{-1} \in S \). A Cayley graph $\Cay(G,S)$ is connected if and only if $G=\langle S\rangle$.

During the twentieth century, Cayley graphs became central objects in algebraic graph theory. They provide a natural way to encode the structure of a group into a graph, allowing methods from linear algebra, combinatorics, and representation theory to be applied to group-theoretic problems. The spectrum of a graph, namely the set of eigenvalues of its adjacency matrix, has proven particularly important because it captures structural information about connectivity, symmetry, expansion properties, and random walks \cite{LZ}.

Given a graph $\Gamma$, a root of  the characteristic polynomial of the adjacency matrix of $\Gamma$ is called an eigenvalue of $\Gamma$.  A graph is called \textit{integral} if all of its eigenvalues are integers. Integral graphs were first systematically studied by Harary and Schwenk in the 1970s \cite{HS}, and have since attracted considerable attention due to their rarity and rich algebraic structure. The study of integral Cayley graphs, which were proposed by Abdollahi and Vatandoost \cite{AV},  naturally combines spectral graph theory with finite group theory, leading to the question of determining which groups exclusively give rise to integral Cayley graphs. This inquiry initiated the theory of Cayley integral groups, which refers to finite groups for which every undirected Cayley graph is integral.

Significant progress was achieved through the work of Klotz and Sander, who classified all finite abelian Cayley integral groups \cite{KS}. Subsequently, Abdollahi and Jazaeri \cite{AJ}, and independently Ahmady, Bell, and Mohar \cite{Azhvan}, completed the classification in the non-abelian case, providing a complete description of all finite Cayley integral groups. These results revealed that the property of being Cayley integral is highly restrictive and is satisfied by only a small collection of finite groups.

%Motivated by these classification theorems, Estélyi and Kovács \cite{EK} introduced a broader framework by examining groups whose Cayley graphs are integral solely under bounded valency conditions. For a positive integer \( k \), a finite group \( G \) belongs to the class \( \mathcal{G}_k \) if every undirected Cayley graph \( \Cay(G,S) \) with a generating set \( S \) that satisfies \( |S| \leq k \) is integral. This generalization shifts the focus from the global requirement that all Cayley graphs be integral to the more flexible condition that only those with bounded degree possess integral spectra.

%The investigation into the class \( \mathcal{G}_k \) has unveiled several surprising phenomena. Specifically, it has been demonstrated that for \( k \geq 6 \), the class \( \mathcal{G}_k \) coincides with the class of Cayley integral groups. This finding indicates that the integrality of Cayley graphs with relatively small valency imposes significant algebraic constraints on the underlying group. Moreover, the exceptional cases \( k = 4 \) and \( k = 5 \) display additional families of groups not found in the classical classification, highlighting the subtle interactions between group structure and spectral properties.

M\"{o}nius, Steuding, and Stumpf \cite{MSS} introduced the concepts of splitting fields and the algebraic degrees of graphs to explore which graph properties prevent integral eigenvalues. For a graph \( \Gamma \), its \textit{splitting field} \( \text{SF}(\Gamma) \) is defined as the smallest field extension of the rational number field \( \mathbb{Q} \) that encompasses all the eigenvalues of the adjacency matrix of \( \Gamma \). The extension degree \( [\text{SF}(\Gamma) : \mathbb{Q}] \) is termed the \textit{algebraic degree} of \( \Gamma \) and is denoted by \( \text{deg}(\Gamma) \). Since eigenvalues of graphs are all algebraic integers, a graph \( \Gamma \) is considered integral if and only if \( \text{deg}(\Gamma) = 1 \). Recently, various studies have been conducted to determine the algebraic degrees of Cayley graphs and their generalizations \cite{HLM,LM,SM,WAF,WYF}. In \cite{AAFW}, a generalization of integral graphs is proposed, whereby a graph \( \Gamma \) is regarded as \textit{\( k \)-integral} if \( \text{deg}(\Gamma) = k \). In the same paper, the authors examine \( 2 \)-integral Cayley graphs over finite groups \( G \) with respect to Cayley sets that are unions of conjugacy classes of \( G \). Among other general results, they completely classify all finite abelian groups with a connected \( 2 \)-integral Cayley graph that has valency 2, 3, 4, and 5.

In this paper, we define Cayley \( k \)-integral groups as follows:
\begin{defi}
	A finite non-identity group $G$ is called \textrm{Cayley $k$-integral}, if for every inverse-closed subset $S$ of $G\setminus\{1\}$, the algebraic degree of $\Cay(G,S)$ is at most $k$.
\end{defi}

Clearly Cayley $1$-integral groups are exactly Cayley integral groups. These groups are classified:
\begin{thm}$($\cite[Theorem 1.2]{AJ} and \cite[Theorem 4.2]{Azhvan}$)$\label{azhv}
	The only Cayley integral groups are $\Bbb Z_2^n\times\Bbb Z_3^m$, $\Bbb Z_2^n\times\Bbb Z_4^m$, $Q_8\times\Bbb Z_2^n$, $S_3$ and $\textrm{Dic}_{12}$, where $m,n$ are arbitrary non-negative integers, $Q_8$ is the quaternion group of order $8$, and $\textrm{Dic}_{12}$ is the dicyclic group of order $12$.
\end{thm}
In this paper, we completely classify all finite Cayley $2$-integral groups:

\begin{thm}\label{main}
	A finite group $G$ is Cayley $2$-integral if and only if $G$ is one of the following groups

\begin{itemize}
	\item[(1)] $\Bbb Z_2^r\times\Bbb Z_5^s$, $r,s\geq 0$,
	\item[(2)] $\Bbb Z_2^r\times\Bbb Z_4^s\times \Bbb Z_8^t$, $r,s,t\geq 0$,
	\item[(3)] $\Bbb Z_2^r\times\Bbb Z_3^s\times \Bbb Z_{12}^t$, $r,s,t\geq 0$,
	\item [(4)] $Q_8\times\Bbb Z_2^n$, $n\geq 0$,
	\item[(5)] $D_{2n}=\langle a,b\mid a^n=b^2=(ab)^2=1\rangle$, $n=3,4,5,6$,
		\item[(6)] $Dic_{4n}=\langle a,b\mid a^{2n}=1, b^2=a^n, bab^{-1}=a^{-1}\rangle$, $n=3,5$,
	%\item[(5)] $\langle a,b\mid a^4=1, b^2=a^2, a^b=a^{-1}\rangle\cong Q_8$,
	\item[(7)] $\langle a,b,c\mid a^4=b^2=c^2=1,ab=ba,bc=cb,cac=ab\rangle\cong(\Bbb Z_4\times\Bbb Z_2)\rtimes\Bbb Z_2$,
	\item[(8)] $\langle a,b\mid a^4=b^4=1, bab^{-1}=a^3\rangle\cong\Bbb Z_4\rtimes\Bbb Z_4$,
	\item[(9)] $\langle a,b\mid a^8=b^2=1, bab=a^5\rangle\cong\Bbb Z_8\rtimes\Bbb Z_2$,
	\item[(10)]$\langle a,b\mid a^8=1,a^4=b^2,a^b=a^{-1}\rangle\cong Q_{16}$,
	\item[(11)] $\langle a,b,c\mid a^4=b^2=c^2=(ab)^2=1, ac=ca, bc=cb\rangle\cong D_8\times\Bbb Z_2$,
	%\item[(11)] $\langle a,b,c\mid a^4=c^2=1, b^2=a^2, a^b=a^{-1}, ac=ca,bc=cb\rangle\cong Q_8\times\Bbb Z_2$
	\item[(12)] $\langle a,b,c\mid a^4=b^2=1,a^2=c^2, bab^{-1}=a^{-1}, bc=cb,ac=ca\rangle\cong (\Bbb Z_4\times\Bbb Z_2)\rtimes\Bbb Z_2$,
	\item[(13)] $G_{32,50}=\langle a,b,c,d\mid a^2=b^4=d^2=c^2b^2=cc^{b^{-1}}=(ab)^2=dd^b=a^ca=dd^c=a^db^2a=1\rangle\cong (Q_8\times\Bbb Z_2)\rtimes\Bbb Z_2$,
	\item[(14)] $G_{32,35}=\langle a,b,c\mid b^4=c^4=1, a^2=b^{-2}, ab=b^{-1}a,ac=c^{-1}a,bc=cb\rangle\cong\Bbb Z_4\rtimes Q_8$,
	\item[(15)] $G_{32,26}=\langle a,b,c\mid  b^4=a^4=c^4=1, ab=b^{-1}a, ca=ac, cb=bc\rangle\cong\Bbb Z_4\times Q_8$,
	\item[(16)] $G_{32,23}=\langle a,b,c\mid c^2=b^4=a^4=b^ab=cc^a= cc^b=a^2b^{-1}a^2b=(b^{-1}a^{-2}b^{-1})^2=1 \rangle\cong (\Bbb Z_4\rtimes\Bbb Z_4)\times\Bbb Z_2$,
	\item[(17)] $A_4=\langle a,b\mid   b^2=a^3=(ba)^3=1\rangle$,
	\item[(18)] $G_{18,4}=\langle a,b,c\mid  a^2=b^3=c^3=(ab)^2=(ac)^2=c^{-1}c^b=1 \rangle\cong (\Bbb Z_3\times \Bbb Z_3)\rtimes \Bbb Z_2$,
	\item[(19)] $G_{24,4}=\langle a,b,c\mid  c^3=b^2a^2=baba^{-1}=a^4=c^ac=c^{-1}c^b=(a^{-1}c)^2a^{-2}=1 \rangle \cong\Bbb Z_3\rtimes Q_8$,
	\item[(20)] $G_{24,11}=\langle a,b,c\mid  c^3=b^2a^2=baba^{-1}=a^4=c^{-1}c^a=c^{-1}c^b=1\rangle\cong Q_8 \times\Bbb Z_3 $,
	\item[(21)] $G_{24,7}=\langle a,b,c\mid b^2=c^3=a^4=c^ac=bb^a=b^cb=1 \rangle\cong\Bbb Z_2\times(\Bbb Z_3\rtimes \Bbb Z_4)$,
	\item[(22)] $G_{36,7}=\langle a,b,c\mid  b^3=c^3=a^4=b^ab=c^ac= c^{-1}c^b=1 \rangle \cong(\Bbb Z_3\times \Bbb Z_3)\rtimes \Bbb Z_4$.
\end{itemize}
	\end{thm}
	
\section{Preliminaries and notation}
In this paper, graphs are finite, undirected, loop-free and without multiple edges. Also the groups are finite. Given a group $G\neq 1$ and an inverse-closed subset $S$ of $G\setminus\{1\}$, the Cayley graph
$\Cay(G,S)$ has $|G:\langle S\rangle|$ components each of which is isomorphic to  $\Cay(\langle S\rangle, S)$. In particular, $\Cay(G,S)$ is connected if and only if  $G=\langle S\rangle$. By \cite[Theorem 2.6]{Cvet}, the following result is clear.
\begin{lem}\label{regular}
	Let $\Gamma$ be a regular graph with $\deg(\Gamma)=k$. Then $\deg(\Gamma^c)=k$, where $\Gamma^c$ is the complement of $\Gamma$. 
\end{lem}

By the following corollary, to study the Cayley $k$-integral groups, it  suffices to study their connected Cayley graphs.
\begin{core}\label{connected}
	A group $G$ is Cayley $k$-integral if and only if every connected Cayley graph of $G$ has algebraic degree at most $k$.
\end{core}
\begin{proof}
	One direction is obvious. Conversely, towards a contradiction, suppose that every connected Cayley graph of $G$ has algebraic degree at most $k$, but there exists a subset $S$ of $G$ such that $\Gamma:=\Cay(G,S)$ has algebraic degree greater than $k$. Then $\Cay(G,S)$ is disconnectd. Hence $\Gamma^c=\Cay(G,G\setminus( S\cup\{1\}))$ has algebraic degree at most $k$, which contradicts Lemma \ref{regular}. 
\end{proof} 

The proof of the following result is similar to \cite[Lemma 4.3]{Azhvan}, but we give it for completeness.
\begin{lem}\label{sbg-qoutient}
	Let $G$ be a Cayley $k$-integral group. Then every subgroup and every quotient group of $G$ is Cayley $k$-integral. 
\end{lem}
\begin{proof}
Let $H\leq G$ and $S$ be an inverse-closed subset of $H\setminus\{1\}$. Then $\Cay(G,S)$ is the disjoint union of  $|G:H|$ copeis of $\Cay(H,S)$. This implies that $H$ is also Cayley $k$-integral.

Now let $N\unlhd G$ and $\Gamma=\Cay(G/N,S)$ for some inverse-closed subset $S\neq N$ of $G/N$. Let $\pi:G\rightarrow G/N$ be the canonical epimorphism and $T:=\pi^{-1}(S)$. Then $A_G=A_{G/N}\otimes J_n$, where $A_G$ and $A_{G/N}$ are adjacency matrices of $\Cay(G,T)$ and $\Cay(G/N,S)$, respectively, $n=|N|$ and $J_n$ is the all ones $n\times n$ matrix. Since eigenvalues of $J_n$ are $0$ with multiplicity $n-1$ and $n$ with multiplicity $1$, the eigenvalues of $\Cay(G,T)$ are $0$ with multiplicity $(n-1)|G|/|N|$, and $n\lambda_i$,  with multiplicity $m_i$, where distinct eigenvalues of $\Cay(G/N,S)$ are $\lambda_1,\ldots,\lambda_t$ and $m_i$ is the multiplicity of $\lambda_i$. Hence $\textrm{SF}(\Cay(G,T))=\Bbb Q(n\lambda_1,\ldots,n\lambda_t)=\Bbb Q(\lambda_1,\ldots,\lambda_t)=\textrm{SF}(\Cay(G/N,S))$, since $n$ is an integer. This means that $G/N$ is Cayley $k$-integral.
\end{proof}

Recall that given a positive non-square integer $m$, the field $\Bbb Q(\sqrt{m})=\{q_1+q_2\sqrt{m}\mid q_1,q_2\in\Bbb Q\}$ is a field extension of $\Bbb Q$ of degree $2$ (a $\Bbb Q$-vector space over $\Bbb Q$ of dimension $2$). Thus given a graph $\Gamma$, $\deg(\Gamma)=2$ if and only if $\textrm{SF}(\Gamma)=\Bbb Q(\sqrt{m})$ for some fixed positive non-square integer $m$. Hence a finite group $G$
is Cayley $2$-integral if and only if for every $S=S^{-1}\subseteq G\setminus\{1\}$, each eigenvalue of $\Cay(G,S)$ lies in $\Bbb Q(\sqrt{m})$ for some fixed positive non-square integer $m$. 

By a result of Diaconis and Shahshahani \cite{DS}, see also \cite[Corollary 7]{AT}, one can compute the eigenvalues of a Cayley graph over $G$ by using  irreducible representations of $G$. We briefly recall some basic definitions of representation theory of groups for the convenience of the reader.  Let $G$ be a finite group. A {\em representation} of $G$ is a homomorphism $\rho:G\rightarrow GL(V)$ for some finite dimensional vector space $V$. Usually, we write $\rho_g$ for $\rho(g)$ and $\rho_g(v)$ for the action of $\rho_g$ on $v\in V$. The dimension of $V$ is called the {\em degree} of $\rho$.  Two representations $\varphi:G\rightarrow GL(V)$ and $\psi:G\rightarrow GL(W)$ are said to be {\em equivalent} if there exists an isomorphism $\tau:V\rightarrow W$ such that $\psi_g= \tau\varphi_g\tau^{-1}$ for all $g\in G$. For a representation $\rho:G\rightarrow GL(V)$, a subspace $W\leq  V$ is said to be {\em $G$-invariant} if for all $g\in G$ and $w\in W$, one has $\rho_g(w)\in W$. The representation
$\rho$ is said to be {\em irreducible} if the only $G$-invariant subspaces of $V$ are $\{0\}$ and $V$. The reader is referred to \cite{JL} for more details on representations of groups.

\begin{rem}\label{rep}
	Let $\textrm{Irr}(G)=\{\rho_1,\ldots,\rho_n\}$ be the set of all inequivalent irreducible complex representations of $G$. If $\lambda$ is an eigenvalue of $\Cay(G,S)$, then $\lambda$ is an eigenvalue of $\rho_i(S):=\sum_{s\in S}\rho_i(s)$, for some $1\leq i\leq n$, and conversely, each eigenvalue of $\rho_i(S)$ appears as an eigenvalue of $\Cay(G,S)$, \cite[Corollary 7]{AT}. 
\end{rem}

Our notations are standard and mainly taken from \cite{Robinson}, but for the reader's convenience we recall some of them as follows:
\begin{itemize}
	\item $\exp(G)$: the exponent of $G$.
	\item $\langle a\rangle$: the cyclic group generated by $a$.
	\item $\Bbb Z_n$: the additive group of integers modulo $n$.
	\item $C_H(K)$: the set of elements of $H$ commute with all elements of $K$.
	\item $Z(G)$: the center of $G$.
	\item $H\times K$: the direct product of $H$ and $K$.
	\item $H\rtimes K$: the semidirect product of $H$ by $K$.
	\item $D_{2k}$: the dihedral group of order $2k$.
	\item $Q_8$: the quaternion group of order $8$.
	\item $S_n$: the symmetric group of degree $n$.
	\item $A_n$: the alternating group of degree $n$.
	\item $\Phi(G)$: the Frattini subgroup of $G$.
	\item $[H,K]$: the group generated by all $[h,k]:=h^{-1}k^{-1}hk$, where $h\in H$ and $k\in K$.
	\item $G'$: the derived subgroup $[G,G]$ of $G$.
	\item $G_{n,r}$: the small group with IdGroup $[n,r]$ based on GAP library.
	\item $x^y$: the element $y^{-1}xy$.
	\item $x^{\pm k}\in S$: means $x^k, x^{-k}\in S$, where $k\geq 1$.
\end{itemize}

\section{Finite abelian Cayley $2$-integral groups}
\begin{lem}\label{exp}
	Let $G$ be a finite abelian group, and $\exp(G)\in\{2,3,4,5,6,8,10,12\}$. Then $G$ is Cayley $2$-integral.
\end{lem}
\begin{proof}
	Let $n=\exp(G)$. If $n=2$, then $G$ is an elementary abelian $2$-group and so is Cayley $2$-integral \cite[Theorem 4.2]{Azhvan}. Hence we may assume that $n\geq 3$. Then $\varphi(n)=2$ or $4$, where $\phi$ is the Euler function. Let $S$ be an inverse-closed subset of $G$ not including the identity element of $G$ and $\Gamma=\Cay(G,S)$. Since $n>2$ and $S$ is inverse-closed, $m:=|\{k\in\Bbb Z_n^*\mid S^k=S\}|\geq 2$. Hence, by \cite[Corollary 3.10]{WAF}, we have  $\deg(\Gamma)=\frac{\varphi(n)}{m}\leq 2$, which means $G$ is Cayley $2$-integral.
\end{proof}

%\begin{core}\label{ab}
%	For any $r,s,t\geq 0$, the  following abelian groups are Cayley $2$-integral:
%	\begin{itemize}
	%%	\item $\Bbb Z_2^r\times\Bbb Z_3^s$
	%%	\item $\Bbb Z_2^r\times\Bbb Z_4^s$
	%%	\item $\Bbb Z_5^k$, $k\geq 1$
	%%	\item $\Bbb Z_8^k$, $k\geq 1$
	%%	\item $\Bbb Z_{10}^k$, $k\geq 1$
	%%	\item $\Bbb Z_{12}^k$, $k\geq 1$
%		\item $\Bbb Z_2^r\times\Bbb Z_5^s$%\times\Bbb Z_{10}^t$
%		\item $\Bbb Z_2^r\times\Bbb Z_4^s\times \Bbb Z_8^t$ 
%		\item $\Bbb Z_2^r\times\Bbb Z_3^s\times \Bbb Z_4^t$%\times\Bbb Z_6^m\times \Bbb Z_{12}^n$.
%	\end{itemize}
%\end{core}

\begin{core}\label{order}
	Let $G\neq 1$ be a finite cyclic group. Then $G$ is  Cayley $2$-integral if and only if $G\cong\Bbb Z_n$, where $n\in\{2,3,4,5,6,8,10,12\}$.
\end{core}
\begin{proof}
	Let $G=\langle g\rangle\cong\Bbb Z_n$ be Cayley $2$-integral and $n>2$. Then $\Cay(G,\{g,g^{-1}\})$ is integral or $2$-integral. In the first case, $n=3,4,6$ \cite[Lemma 2.7]{AV} and in the latter case $n=5,8,10,12$ \cite[Theorem 5.1]{AAFW}. Thus $n\in\{2,3,4,5,6,8,10,12\}$. The converse direction follows from Lemma \ref{exp}.
\end{proof}

\begin{core}\label{2,3,5}
	If $G$ is a finite Cayley $2$-integral group, then $G$ is a $\{2,3,5\}$-group. In particular, the order of every non-identity element of $G$ is $2,3,4,5,6,8,10$ or $12$.
\end{core}
\begin{proof}
	It is a direct consequence of Lemma \ref{sbg-qoutient} and Corollary \ref{order}.
\end{proof}

\begin{pro}\label{baseab}
The only finite abelian Cayley $2$-integral groups are 
\begin{itemize}
	\item[(1)] $\Bbb Z_2^r\times\Bbb Z_5^s$,
	\item[(2)] $\Bbb Z_2^r\times\Bbb Z_4^s\times \Bbb Z_8^t$,
	\item[(3)] $\Bbb Z_2^r\times\Bbb Z_3^s\times \Bbb Z_{12}^t$,
\end{itemize}
where $r,s,t$ are arbitrary non-negative integers.
\end{pro}
\begin{proof}
	By Fundamental Theorem of Finite Abelian Groups, we have $G\cong\Bbb Z_{n_1}\times\cdots\times\Bbb Z_{n_t}$, where $t\geq 1$ and $n_i$s are positive integers greater than $1$ such that $n_1|n_2|\cdots|n_t$.
	Let $G$ be a Cayley $2$-integral group. Then, by Lemma \ref{sbg-qoutient} and Corollary \ref{order}, $n_t\in\{2,3,4,5,6,8,10,12\}$. If $n_t\in\{12,6,3,2\}$, $\{10,5\}$ or $\{8,4\}$ then $G$ is of the form $(3)$, $(1)$ or $(2)$, respectively, for suitable non-negative integers $r,s,t$. The converse direction is clear by Lemma \ref{exp}. This completes the proof.
\end{proof}
\section{Cayley $2$-integrality of some small non-abelian groups}

In this section, we provide a list of some small non-abelian Cayley $2$-integral groups and  non-Cayley $2$-integral groups that will be used in the next section. For a small group $G$, one can check, by a simple program in GAP,  whether $G$ is Cayley $2$-integral by considering all possible inverse-closed generating sets of $G$. However, we give some GAP-free proofs in this paper, by using the following facts:

 It is well-known that $\Cay(G,S)\cong\Cay(G,S^\alpha)$ for all $\alpha\in\Aut(G)$. Thus we have the following remark.
\begin{rem}\label{orbit}
	Let $G$ be a group, $n\in\Bbb N$ and $X_n=\{S\mid 1\notin S=S^{-1}\subseteq G, \langle S\rangle=G,|S|=n\}$. Then $\Aut(G)$ acts on $X_n$ naturally. Let $X_n^{1},\ldots,X_n^{m}$ be all distinct orbits of this action, $x_n^{i}$ be a representative of $X_n^{i}$ for $i=1,\ldots,m$ and $\Gamma_n^i=\Cay(G,x_n^i)$. Then $G$ is a Cayley $2$-integral group if and only if $\deg(\Gamma_n^i)\leq 2$ for all possible $n$ and $i\in\{1,\ldots,m\}$. Also note that we may assume that $n\leq \frac{|G|}{2}$ by Lemma \ref{regular}.
\end{rem} 
\begin{rem}
	 By Remark \ref{rep}, a group $G$ is a Cayley $2$-integral group if and only if there exists a fixed positive non-square integer $m$ such that for every inverse-closed subset $S$ of $G\setminus \{1\}$ and every $\rho\in\textrm{Irr}(G)$, the eigenvalues of $\rho(S)$ lies in $\Bbb Q(\sqrt{m})$. We use this fact, in the next section without referring it.
\end{rem}

\begin{rem}
	Let $\rho$ be an irreducible representation of a group $G$ of degree $d$. By \cite[Proposition 9.14]{JL}, if $z\in Z(G)$ is an involution, then $\rho(z)=I_d$ or $-I_d$, where $I_d$ is the $d\times d$ identity matrix.  If $S$ contains $z$ then the splitting field of   $\Cay(G, S)$ is the same as the splitting field  of $\Cay(G, S\setminus\{z\})$, by \cite[Corollary 7]{AT}. 
	We use this fact  without further  referring.
\end{rem}

\subsection{Some non-abelian Cayley $2$-integral groups $G$ with $|G|\leq 36$}

By Theorem \ref{azhv}, the non-abelian groups $S_3$ and $Dic_{12}$ are Cayley integral and so Cayley $2$-integral. The following lemma is well-known. We use it in the next result.
\begin{lem}\label{aut}
	Let $D_{2n}=\langle a,b\mid
	a^n=b^2=(ab)^2=1\rangle$, $1\leq l\leq n-1$,
	$0\leq s\leq n-1$ and $(l,n)=1$. Then the automorphisms of $D_{2n}$
	are of the form
	\begin{eqnarray*}
		\sigma_{l,s}:&&D_{2n}\rightarrow D_{2n}\\
		&&a^i\mapsto a^{il}\\
		&&a^ib\mapsto a^{il+s}b.
	\end{eqnarray*}
	Furthermore,
	$\Aut(\langle a\rangle)=\{\sigma_l:\langle a\rangle\rightarrow\langle a\rangle \mid (a^{i})^{\sigma_l}=a^{il}, (l,n)=1\}$.
\end{lem}

\begin{lem}\label{dihedral}
	The dihedral group $D_{2k}$, $k\geq 3$, is Cayley $2$-integral for  $k=3,4,5,6$.
\end{lem}
\begin{proof}
	Let	$k\in\{3,4,5,6\}$. Since $D_{6}\cong S_3$ is a Cayley integral group,  we may assume that $k\in\{4,5,6\}$. Let  $S=S^{-1}\subseteq D_{2k}\setminus\{1\}$  and  $\Gamma=\Cay(D_{2k},S)$. We may assume that $D_{2k}=\langle S\rangle$. We use Remark \ref{orbit}, and give all the possible representatives of the action of $\Aut(D_{2k})$ on $X_n$'s, $1\leq n\leq k$. For example, we do this for $k=4$. One can prove the remaining cases similarly. Let $k=4$. Then $\Aut(D_{2k})$ has one orbit on $X_2$, two orbits on $X_3$ and three orbits on $X_4$, with representatives
	\begin{eqnarray*}
		&&S_1=\{ab,b\}, S_2=\{ab,a^2b,a^3b\}, S_3=\{a^2,a^2b,a^3b\}, S_4=\{b,ab,a^2b,a^3b\},\\
		&& S_5=\{a^3b,a^2b,a,a^3\}, S_6=\{a^3b,a^2b,ab,a^2\}.
	\end{eqnarray*}
	
	Let $\Gamma_i=\Cay(G,S_i)$. Then $\deg(\Gamma_{2i-1})=2$ and $\deg(\Gamma_{2i})=1$ for $i\in\{1,2,3\}$, which means $D_{8}$ is Cayley $2$-integral.
\end{proof}
\begin{lem}\label{q8z3}
	$Q_8\times\Bbb Z_3$ is Cayley $2$-integral.
\end{lem}
\begin{proof}
	We know that $H:=\langle x,y\mid x^4=1, x^2=y^2, y^{-1}xy=x^{-1}\rangle\cong Q_8$ has $5$ inequivalent irreducible representations $\rho_1,\ldots,\rho_5$, and exactly one of them, say $\rho_5$ has degree $2$ and the others are all linear. More precisely, 
	\begin{eqnarray*}
		&&\rho_1:x\mapsto 1,~~y\mapsto 1\\
		&&\rho_2: x\mapsto -1,~~y\mapsto 1\\
		&&\rho_3: x\mapsto 1,~~y\mapsto -1\\
		&&\rho_4: x\mapsto -1,~~y\mapsto -1
	\end{eqnarray*}
	and 
	\begin{eqnarray*}
		\rho_5:x\mapsto\begin{bmatrix}
			0&-1\\
			1&0
		\end{bmatrix},~~y\mapsto\begin{bmatrix}
			\textbf{i}&0\\
			0&-\textbf{i}
		\end{bmatrix},
	\end{eqnarray*}
	where $\textbf{i}=\sqrt{-1}$. Furthermore, the conjugacy classes of $H$ are $\{1\}$, $\{x,x^{-1}\}$, $\{y,y^{-1}\}$, $\{x^2\}$ and $\{xy,xy^3\}$. Also $K:=\langle a\rangle\cong \Bbb Z_3$ has three inequivalent irreducible represenations $\psi_0,\psi_1,\psi_2$, where $\psi_l(a^j)=e^{\frac{2\pi \textbf{i}lj}{3}}$   for $l=0,1,2$. Let $G=H\times K$. Then the inequivalent irreducible representations of $G$ are
	$\rho_i\otimes\psi_j$, where $i=1,\ldots,5$ and $j=0,1,2$ and $(\rho_i\otimes\psi_j)(h,k)=\rho_i(h)\psi_j(k)$ for all $(h,k)\in G$.
	
	Let $S$ be an inverse-closed subset of $G$, $(1,1)\notin S$ and $\Gamma=\Cay(G,S)$. Since $(1,1)\notin S$ and $S=S^{-1}$, there exists a positive integer $m$ such that $S=(\bigcup_{l=1}^m S_l)\cup T$, where $T=\varnothing$ or $\{(x^2,1)\}$, and $S_l=\{(z_l,w_l),(z_l^{-1},w_l^{-1})\}$ is a $2$-element set for each $l$.  Let $\lambda$ be an eigenvalue of $\Gamma$. Then, $\lambda$ is an eigenvalue of $B_{i,j}:=(\rho_i\otimes\psi_j)(S)=\sum_{s\in S}(\rho_i\otimes\psi_j)(s)$ for some $i\in\{1,\ldots,5\}$ and $j\in\{0,1,2\}$. Since for each $1\leq i\leq 4$, $\rho_i(x^2)=1$ and $\rho_5(x^2)=-I_{2\times 2}$, we may assume that $T=\varnothing$, or equivalently $S=\bigcup_{l=1}^m S_l$. Let $B_{i,j,l}=(\rho_i\otimes\psi_j)(S_l)$. Then $B_{i,j}=\sum_{l=1}^mB_{i,j,l}$. If $i\neq 5$ then $B_{i,j,l}(S_l)=\rho_i(z_l)(\psi_j(w_l)+\psi_j(w_l^{-1}))=[b]$ is a $1\times 1$ matrix with entry $b\in\{-2,-1,1,2\}$. Hence, if $i\neq 5$ then for each $j$, $B_{i,j}$ is a $1\times 1$ matrix with an integer entry and so in this case $-|S|\leq \lambda\leq |S|$ is an integer. Now we assume that $i=5$ and consider $B_{5,j}$ where $j=0,1,2$. We have
	$B_{5,0,l}=\rho_5(z_l)+\rho_5(z_l^{-1})$. By considering all possibilities for $z_l$, we have $B_{5,0,l}=kI_2$, where $k\in\{-2,0,2\}$ and $I_2$ is the $2\times 2$ identity matrix. This means that if $(i,j)=(5,0)$ then $\lambda$ is an integer. Also $B_{5,1,l}=\rho_5(z_l)\psi_1(w_l)+\rho_5(z_l^{-1})\psi_1(w_l^{-1})$ is equal to one of the matrices $\pm I_2$, $O_2$, $\pm \textbf{i}\sqrt{3}\rho_5(x)$, $\pm \textbf{i}\sqrt{3}\rho_5(y)$ or $\pm \textbf{i}\sqrt{3}\rho_5(xy)$, where $O_2$ is the all zero $2\times 2$ matrix. So there exist $\alpha,\beta,\gamma,\delta\in\Bbb Z$ such that
	\begin{eqnarray*}
		B_{5,1}&=&\alpha I_2+\textbf{i}\beta\sqrt{3}\rho_5(x)+\textbf{i}\gamma\sqrt{3}\rho_5(y)+\textbf{i}\delta\sqrt{3}\rho_5(xy)\\
		&=&\alpha I_2+\sqrt{3}\begin{bmatrix}
			-\gamma&-\beta\textbf{i}-\delta\\
			\beta\textbf{i}-\delta&\gamma
		\end{bmatrix}.
	\end{eqnarray*}
	Since the right-hand matrix in the above equality has eigenvalues $\pm\sqrt{\beta^2+\gamma^2+\delta^2}$, the eigenvalues of $B_{5,1}$ are
	\[\alpha\pm \sqrt{3(\beta^2+\gamma^2+\delta^2)}.\] 
	Since $\psi_2(a^{i})=\psi_1(a^{-i})$, for $i=0,1,2$, we have $B_{5,2,l}=\rho_5(z_l)\psi_1(w_l^{-1})+\rho_5(z_l^{-1})\psi_1(w_l)$. If $z_l\neq 1,x^2$ then $\rho_5(z_l)=-\rho_5(z_l^{-1})$ and $B_{5,2,l}=-B_{5,1,l}$. If $z_l=1$ then $\rho_5(z_l)=\rho_5(z_l^{-1})=I_2$ and $B_{5,2,l}=B_{5,1,l}$. If $z_l=x^2$ then $\rho_5(z_l)=\rho_5(z_l^{-1})=-I_2$ and again $B_{5,2,l}=B_{5,1,l}$. This means that the splitting field of $B_{5,2}$ and $B_{5,1}$ are same, and so $\Bbb {SF}(\Gamma)=\Bbb Q(\sqrt{3(\beta^2+\gamma^2+\delta^2)})$. This proves that $Q_8\times\Bbb Z_3$ is Cayley $2$-integral as desired.
\end{proof}

\begin{lem}\label{32yes}
	The following non-abelian groups of order 32 are Cayley $2$-integral:
	\begin{itemize}
		\item[(1)] $G_{32,50}=\langle a,b,c,d\mid a^2=b^4=d^2=c^2b^2=cc^{b^{-1}}=(ab)^2=dd^b=a^ca=dd^c=a^db^2a=1\rangle$,
		\item[(2)] $G_{32,35}=\langle a,b,c\mid b^4=c^4=1, a^2=b^{-2}, ab=b^{-1}a,ac=c^{-1}a,bc=cb\rangle$,
		\item[(3)] $G_{32,26}=\langle a,b,c\mid  b^4=a^4=c^4=1, ab=b^{-1}a, ca=ac, cb=bc\rangle\cong\Bbb Z_4\times Q_8$,
		\item[(4)] $G_{32,23}=\langle a,b,c\mid c^2=b^4=a^4=b^ab=cc^a= cc^b=a^2b^{-1}a^2b=(b^{-1}a^{-2}b^{-1})^2=1 \rangle$.
		%	\item[(5)] ????
	\end{itemize}
\end{lem}
\begin{proof}
	$(1)$ Let $S$ be an inverse-closed subset of $G=G_{32,50}$ and $1\notin S$. Then
	$S$ is a union of the following sets

	$$\{ a \}, \{ b^2 \}, \{ d \}, \{ ab \}, \{ ab^2 \},\{ db^2 \},	\{ abc \},\{ ab^3 \},\{ acd \}, \{ abcb^2 \}, \{ acdb^2 \},$$
	$$ 
	\{ b, b^3\}, \{ c, cb^2 \},\{ ac, acb^2 \},
	\{ ad, adb^2 \}, \{ bc, bcb^2 \}, \{ bd, bdb^2 \},$$
	$$\{ cd, cdb^2 \}, \{ abd, abdb^2 \},   \{ bcd, bcdb^2 \},
	\{ abcd, abcdb^2 \}. $$
	By GAP, we see that $G$ has exactly $17$ irreducible representations, one of them, say $\varphi$, has degree $4$ and the others have degrees $1$ and are integer-valued. We claim that the eigenvalues of $\varphi(S)$ are in $\Bbb Q[\sqrt{m}]$ for some fixed positive integer $m$, which  proves $\Cay(G,S)$ is Cayley $2$-integral.  By GAP, we have
	\begin{equation*}
		\varphi(a)=\begin{bmatrix}
			0& 0&1&0 \\
			0&0&0&1\\
			1&0&0&0\\
			0&1&0&0
		\end{bmatrix},~~	\varphi(b)=\begin{bmatrix}
			0& 1&0&0 \\
			-1&0&0&0\\
			0&0&0&-1\\
			0&0&1&0
		\end{bmatrix},
	\end{equation*}
	\begin{equation*}
		\varphi(c)=\begin{bmatrix}
			-\textbf{i}& 0&0&0 \\
			0&\textbf{i}&0&0\\
			0&0&-\textbf{i}&0\\
			0&0&0&\textbf{i}
		\end{bmatrix},~~	\varphi(d)=\begin{bmatrix}
			-1& 0&0&0 \\
			0&-1&0&0\\
			0&0&1&0\\
			0&0&0&1
		\end{bmatrix},
	\end{equation*}
	where $\textbf{i}=\sqrt{-1}$. Then $\varphi(b^2)=-I$, where $I$ is the $4\times 4$ identity matrix. Hence if $X$ is a set of size $2$ chosen from the above sets, then $\varphi(X)=O$, where $O$ is the $4\times 4$ zero matrix. Let $S=S_1\cup S_2$, where $S_j$ is a union of the above sets of size $j$, $j=1,2$. Then $\varphi(S)=\varphi(S_1)+\varphi(S_2)=\varphi(S_1)+O=\varphi(S_1)$. Since $\varphi(b^2)=-I$, we may assume that $b^2\notin S_1$. Hence $S_1$ is a union of the sets 
	$$\{ a \},  \{ ab^2 \},\{ d \}, \{ db^2 \}, \{ ab \}, \{ ab^3 \},	\{ abc \}, \{ abcb^2 \}, \{ acd \},  \{ acdb^2 \}.$$
	Since $\varphi(x)+\varphi(xb^2)=O$ for any $x\in G$, we may assume that $S_1$ has at most $5$ elements and if $x\in S_1$ then $xb^2\notin S_1$. By considering all possibilities for $S$ and by a tedious computation, one can see that the characteristic polynomial of $\varphi(S)$ is one of the following
	\[ (x-1)^2(x+1)^2,~ (x^2-2)^2,~(x^2-3)^2, (x-2)^2(x+2)^2,~(x^2-5)^2,\]
	which proves our claim and completes the proof.
	
	$(2)$. One can check that $Z(G)=\{1,a^2,c^2,a^2c^2\}$.  Furthermore, the only non-identity elements of $G$ of order $2$ are $a^2,c^2,a^2c^2$.  Let $S$ be an inverse-closed subset of $G=G_{32,35}$ and $1\notin S$. Then
	$S$ is a union of the following sets
	$$\{a^2\},~\{c^2\},~\{a^2c^2\},$$
	$$
	\{a, a^3\},
	\{b, ba^2\},
	\{ c, c^3\},
	\{ab, aba^2\},
	\{ac, aca^2\},
	\{ ac^2, a^3c^2\},
	\{ bc, bca^2c^2\},$$
	$$
	\{bc^2, ba^2c^2\},
	\{ca^2, ca^2c^2\},
	\{abc, abca^2\},
	\{abc^2, aba^2c^2\},
	\{ac^3, ac^3a^2\},
	\{bca^2, bc^3\},
	\{abc^3, abc^3a^2\}.
	$$
	
	Clearly, if $\{x_1,x_2\}$ is a set of the above list, then $x_2=x_1y$, where $y\in\{a^2,c^2,a^2c^2\}$. Let $\rho$ be an irreducible representation of $G$ of degree $d$, and $S=X\cup Y$, where $X$ and $Y$ are a union of the sets of size $2$ and size $1$, chosen from the above list,  respectively. Then $\rho(S)=\rho(X)+\rho(Y)$. Since the elements of $Y$ are involutions, $\rho(Y)$ is an integer multiple of  $I$, where $I$ is the $d\times d$ identity matrix. So it is enough to find the eigenvalues of $\rho(X)$. By considering all possibilities for $X$ and all irreducible representations of $G$, one can see that, the characteristic polynomial of $\Cay(G,X)$ is of the form $f_1(x)f_2(x)$ where $f_1(x)$ is a product of factors of degree $1$, and either $f_2(x)=1$ or $f_2(x)$ is one of the followings:
	\[(x^2-8)^4,~(x^2-20)^2,~(x^2-32)^2,(x^2+4x-16)^2,~(x^2-4x-16)^2,\]
	\[(x^2-4x-28)^2,(x^2+4x-28)^2,~(x^2-4x-4)(x^2+4x-4).\]
	This implies that $G$ is Cayley $2$-integral.
	
	$(3)$ One can see that $Z(G)=\{1,c,b^2,a^2,cb^2,ca^2,b^2a^2,cb^2a^2\}$.
	Let $S$ be an inverse-closed subset of $G=G_{32,26}$ and $1\notin S$. Then
	$S$ is a union of the following sets
	\[\{a^2\},\{b^2\},\{a^2b^2\},\]
	\[\  \{a, a^3\}, \{b, b^3\} ,\{ c, cb^2a^2 \},\{ ab, aba^2 \},\{ ac, acb^2 \},
	\{ ab^2, ab^2a^2 \},\{ bc, bca^2 \},\]
	\[\{ ba^2, b^3a^2 \},\{ cb^2, ca^2 \},\{ abc, abcb^2 \},
	\{ ab^3, ab^3a^2 \},\{ a^3c, a^3cb^2 \},\{ b^3c, b^3ca^2 \},\{ a^3bc, a^3bcb^2\}.
	\]
	Clearly, if $\{x_1,x_2\}$ is a set of the above list, then $x_2=x_1y$, where $y\in\{a^2,b^2,a^2b^2\}$. Let $\rho$ be an irreducible representation of $G$ of degree $d$, and $S=X\cup Y$, where $X$ and $Y$ are a union of the sets of size $2$ and size $1$, chosen from the above list,  respectively. Then $\rho(S)=\rho(X)+\rho(Y)$. Since the elements of $Y$ are involutions, $\rho(Y)$ is an integer multiple of $I$, where $I$ is the $d\times d$ identity matrix. So it is enough to find the eigenvalues of $\rho(X)$. By considering all possibilities for $X$ and all irreducible representations of $G$, one can see that, the characteristic polynomial of $\Cay(G,X)$ is of the form $f_1(x)f_2(x)$ where $f_1(x)$ is a product of factors of degree $1$, and either $f_2(x)=1$ or $f_2(x)$ is one of the followings:
	\[(x^2-8)^4,~(x^2-12)^4.\]
	This implies that $G$ is Cayley $2$-integral.
	
	$(4)$
	One can see that $Z(G)=\{1,c,b^2,a^2,cb^2,ca^2,b^2a^2,cb^2a^2\}$.
	Let $S$ be an inverse-closed subset of $G=G_{32,23}$ and $1\notin S$. Then
	$S$ is a union of the following sets
	\[\{c\}, \{b^2\}, \{a^2\}, \{cb^2\}, \{ca^2\}, \{b^2a^2\}, \{cb^2a^2\},\]
	\[ \{ a, a^3 \}, \{ b, b^3 \}, \{ ab, aba^2 \}, \{ ac, aca^2 \}, \{ ab^2, ab^2a^2 \}, \{ bc, bcb^2 \},\]
	\[\{ ba^2, b^3a^2 \}, \{ abc, abca^2 \}, \{ ab^3, ab^3a^2 \},
	\{ acb^2, acb^2a^2 \}, \{ bca^2, bcb^2a^2 \}, \{ abcb^2, abcb^2a^2 \}.
	\]
	Clearly, if $\{x_1,x_2\}$ is a set of the above list, then $x_2=x_1y$, where $y\in\{a^2,b^2,a^2b^2\}$. Let $\rho$ be an irreducible representation of $G$ of degree $d$, and $S=X\cup Y$, where $X$ and $Y$ are a union of the sets of size $2$ and size $1$, chosen from the above list,  respectively. Then $\rho(S)=\rho(X)+\rho(Y)$. Since the elements of $Y$ are involutions, $\rho(Y)$ is an integer multiple of $I$, where $I$ is the $d\times d$ identity matrix. So it is enough to find the eigenvalues of $\rho(X)$. By considering all possibilities for $X$ and all irreducible representations of $G$, one can see that, the characteristic polynomial of $\Cay(G,X)$ is of the form $f_1(x)f_2(x)$ where $f_1(x)$ is a product of factors of degree $1$, and either $f_2(x)=1$ or $f_2(x)$ is one of the followings:
	\[(x^2-8)^4,~(x^2-20)^2,~(x^2-32)^2.\]
	This implies that $G$ is Cayley $2$-integral.
\end{proof}

\begin{lem}\label{remains}
	The following groups are Cayley $2$-integral:
	\begin{itemize}
	
		\item[(1)] $G_{12,3}=\langle a,b\mid   b^2=a^3=(ba)^3=1 \rangle\cong A_4$,
		\item[(2)] $\langle a,b,c\mid a^4=b^2=c^2=1,ab=ba,bc=cb,cac=ab\rangle\cong(\Bbb Z_4\times\Bbb Z_2)\rtimes\Bbb Z_2$,
	\item[(3)] $\langle a,b\mid a^4=b^4=1, bab^{-1}=a^3\rangle\cong\Bbb Z_4\rtimes\Bbb Z_4$,
	\item[(4)] $\langle a,b\mid a^8=b^2=1, bab=a^5\rangle\cong\Bbb Z_8\rtimes\Bbb Z_2$,
	\item[(5)]$\langle a,b\mid a^8=1,a^4=b^2,a^b=a^{-1}\rangle\cong Q_{16}$,
	\item[(6)] $\langle a,b,c\mid a^4=b^2=c^2=(ab)^2=1, ac=ca, bc=cb\rangle\cong D_8\times\Bbb Z_2$,
	%\item[(11)] $\langle a,b,c\mid a^4=c^2=1, b^2=a^2, a^b=a^{-1}, ac=ca,bc=cb\rangle\cong Q_8\times\Bbb Z_2$
	\item[(7)] $\langle a,b,c\mid a^4=b^2=1,a^2=c^2, bab^{-1}=a^{-1}, bc=cb,ac=ca\rangle\cong (\Bbb Z_4\times\Bbb Z_2)\rtimes\Bbb Z_2$,
		\item[(8)] $G_{18,4}=\langle a,b,c\mid  a^2=b^3=c^3=(ab)^2=(ac)^2=c^{-1}c^b=1 \rangle\cong (\Bbb Z_3\times \Bbb Z_3)\rtimes \Bbb Z_2$,
			\item[(9)] $Dic_{20}=\langle a,b\mid a^{10}=1, b^2=a^5, bab^{-1}=a^{-1}\rangle$,
		\item[(10)] $G_{24,4}=\langle a,b,c\mid  c^3=b^2a^2=baba^{-1}=a^4=c^ac=c^{-1}c^b=(a^{-1}c)^2a^{-2}=1 \rangle \cong\Bbb Z_3\rtimes Q_8$,
		\item[(11)] $G_{24,7}=\langle a,b,c\mid b^2=c^3=a^4=c^ac=bb^a=b^cb=1 \rangle\cong\Bbb Z_2\times(\Bbb Z_3\rtimes \Bbb Z_4)$,
		\item[(12)] $G_{36,7}=\langle a,b,c\mid  b^3=c^3=a^4=b^ab=c^ac= c^{-1}c^b=1 \rangle \cong(\Bbb Z_3\times \Bbb Z_3)\rtimes \Bbb Z_4$.

		%\item[(5)] $\langle a,b\mid a^4=1, b^2=a^2, a^b=a^{-1}\rangle\cong Q_8$,
	
	\end{itemize}
\end{lem}
\begin{proof}
	By a similar argument in Lemmas \ref{dihedral}, \ref{q8z3} or \ref{32yes}, one can see that the above groups are Cayley $2$-integral.
\end{proof}

Combining Theorem \ref{azhv} and Lemmas \ref{dihedral}-\ref{remains}, we get the following result:
\begin{core}\label{coremain}
All of the groups $(5)$-$(22)$ given in Theorem \ref{main} are Cayley $2$-integral.
\end{core}

\subsection{Some small non-abelian groups which are not Cayley $2$-integral}
In Table \ref{Table23} we give some small non-abelian groups which are not Cayley $2$-integral.

{\footnotesize
	\begin{table}
		\centering
		\begin{tabular}{|c|c|c|c|}
			%	\toprule
			\hline
			$r$& a presentation of $G$ &elements of S & $\deg(\Gamma)$ \\ \hline
			$2$& $\langle a,b\mid b^{4}=a^{4}=[b^2,a]=[b^3,a^2]=(ba^{3})^4=1\rangle$&  $ a^{\pm 1}, (ba)^{\pm 1}, (a^2ba)^{\pm 1}, b^{\pm 1}$ &$4$\\
			%	$2$&  $b^{4}, a^{4}$& $ b^{2}a^{3}b^{2}a, aba^{2}b^{3}a, (ba^{3})^4$  & $a, a^{3}, b, ba, a^2ba, b^{3}, a^{3}b^{3}a^{2}, a^{3}b^{3}$ &$4$\\
			$4$&$ \langle a,b\mid b^4=a^8=[b^2,a]=a^7(a^5)^b=a(a^3)^{b^3}=1\rangle$&  $ a^{\pm 1}, (a^5b)^{\pm 1}, (a^7b)^{\pm 1}, b^{\pm 1}$ &$4$ \\
			%	$4$&$ b^4, a^8$& $ b^{2}a^{7}b^{2}a, a^{7}b^{3}a^{5}b, aba^3b^{3} $  &$ a, a^{7}, b, a^5b, a^7b, b^{3}, b^{3}a, b^{3}a^{3} $ &$4$ \\
			$5$&$\langle a,b\mid b^2=a^8=[a^6,b]=(ba)^2(ba^{7})^2 =1\rangle$&  $(ba)^{\pm 1}, a^{\pm 1}, (a^2b)^{\pm 1}, a^{\pm 3}$&$4$\\
			%	$5$&$b^2, a^8$&$  aba^{6}ba, (ba)^2(ba^{7})^2 $& $a, ba, a^{7}, a^{7}b, a^2b, ba^{6}, a^3, a^{5}$&4\\
			$6$&$\langle a,b\mid b^2=a^4=(ba^2)^4=(a^3b)^4=[b,a^3]^2=1\rangle$&   $ (a^3ba^2)^{\pm 1}, a^{\pm 1}, (ab)^{\pm 1}, (aba)^{\pm 1}$& $4$\\
			%	$6$&$b^2, a^4$& $ (baba^3)^2, (a^3b)^4, a^2(ba^2)^3b$  &$a, a^3ba^2, a^{3}, ba^3, a^3ba^3, ab, aba, a^{2}ba$& $4$\\
			$7$&$\langle a,b\mid b^2=a^8=(aba)^2=[a,b]^2=1 \rangle $& $a^7ba, b, a^6(ab)^2, a^{\pm 1}, a^6b, (ab)^{\pm 2}$& $6$\\
			%	$7$&$b^2, a^8$& $  (aba)^2, (a^7bab)^2 $ & $a^7ba, b, a^6(ab)^2, a, a^7, a^6b, (ab)^2, (ba^7)^2$& $6$\\
			$8$& $\langle a,b\mid b^4=a^8=a^4b^2=b^{a^{6}}b=(ab)^3ab^{3}=1\rangle$&  $ (a^6b)^{\pm 1}, b^{\pm 1}, a^{\pm 2}, a^{\pm 1}$&$4$\\
			%	$8$& $b^4, a^8$&  $a^4b^2, a^{6}b^{3}a^{6}b, (ab)^3ab^{3}$ &$a, b, a^2, b^{3}a^{2}, a^6b, b^{3}, a^{6}, a^{7}$&$4$\\
			$9$&$ \langle a,b\mid b^2=a^4=[b,a^2]= (ba^{3})^8=1\rangle$& $ b, (ba)^{\pm 1}, a^{\pm 1}, a^2, a^2b, aba $&$8$\\
			%	$9$&$ b^2, a^4$& $ a^{3}ba^2ba^{3}, (ba^{3})^8$ &$ b, a, ba, a^{3}, a^2, a^2b, a^{3}b, aba $&$8$\\
			$10$&$\langle a,b\mid b^4=a^4=[b^{2},a^{3}]=[b,a^2]=ab(ab^3)^3=1\rangle$& $b^{\pm 1}, (ab)^{\pm 1}, (aba)^{\pm 1}, (ab)^{\pm 2}$&$4$\\
			%	$10$&$b^4, a^4$&$  b^2ab^{2}a^{3}, a^{3}b^{3}a^2ba^{3}, a^{3}b^{3}a^{3}b(ab^3)^2 $ & $b, b^3, ab, b^3a^3, aba, a^3b^3a^3, (ab)^2, (b^3a^3)^2$&$4$\\
			$11$&$\langle a,b\mid b^2=a^8= (ab)^4=(a^2)^ba^6=1\rangle$ &   $a^{\pm 1}, (ba)^{\pm 1}, (a^2b)^{\pm 1}, a^{\pm 3} $ & $4$\\
			%	$11$&$b^2, a^8$ & $ a^7ba^2ba^7, (ab)^4 $ & $a, ba, a^7, a^2b, ba^6, a^7b, a^3, a^5 $ & $4$\\
			$12$&$\langle a,b\mid b^4=a^8=a^{7}bab=1\rangle$& $ a^{\pm 2}, a^{\pm 1}, a^{\pm 3}, (ba)^{\pm 1}$&$4$\\
			%	$12$&$b^4,a^8$& $ a^{7}bab $ & $a^2, a^{6}, a, a^{7}, a^3, a^{5}, ba, a^{7}b^{3} $&$4$\\
			$13$&$\langle a,b\mid a^4=b^8=[a^{2},b]=b^{7}(b^3)^{a^{3}}= b(b^{5})^a=1\rangle$& $ a^{\pm 1}, b^{\pm 1}, (ab)^{\pm 1}, (a^3b)^{\pm 1}$&$8$\\
			%	$13$&$a^4,b^8$&$ a^{2}b^{7}a^{2}b, b^{7}ab^3a^{3}, ba^{3}b^{5}a$ &$ a, a^{3}, b, b^{7}, ab, a^3b, b^{7}a, b^{7}a^{3}$&$8$\\
			$14$&$\langle a.b\mid  a^4=b^8= a^{3}bab=1 \rangle$& $ a^{\pm 1}, b^{\pm 1}, (ab)^{\pm 1}, (b^3a)^{\pm 1}$&$4$\\
			%	$14$&$ a^4, b^8 $&$ a^{3}bab$ &$ a, a^{3}, b, b^{7}, ab, b^{7}a^{3}, b^3a, a^{3}b^{5}$&$4$\\
			$15$&$\langle a,b\mid a^8=b^8=b^ab^5=(ab^{7})^2a^2=1\rangle$& $ (a^2ba)^{\pm 1},  a^{\pm 1}, a^{\pm 3}, a^{\pm 2}$&$4$\\
			%	$15$&$a^8,b^8$&$ (ab^{7})^2a^2, b^{7}a^{7}bab^{6} $ &$a, a^2, a^2ba, a^{7}b^{7}a^{6}, a^3, a^{7}, a^{5}, a^{6}$&$4$\\
			$22$&$\langle a,b,c\mid  b^2=c^2=a^4=(cb)^2=(a^{2}b)^2=[c,a]=(ba^{3})^2(ba)^2=1\rangle$& $ b, a^{\pm 1}, a^2b, abac$&$4$\\		
			%	$22$&$ b^2, c^2, a^4$&$ ca^{3}ca, (cb)^2, (a^{2}b)^2, (ba^{3})^2(ba)^2$&$ b, a, a^3, a^2b, abac$&$4$\\
			$24$&$\langle a,b,c\mid b^2=c^4=a^4=[c,a]=[c,b]=(a^{2}b)^2=c^2[b,a]=1\rangle$& $a^{\pm 1},b,(ab)^{\pm 1},(ac)^{\pm 1},a^2b,(abc^3)^{\pm 1}$&$4$\\					
			%	$24$&$b^2, c^4, a^4$&$ c^{3}a^{3}ca, c^{3}bcb, (a^{2}b)^2, c^2ba^{3}ba$ &$a,b,ab,ac,a^3,a^2b,abc^3,a^3bc^2,a^3c^3,a^3bc^3$&$4$\\
			$25$&$\langle a,b,c\mid b^2=a^4=c^4=c^2a^2=(ca^{3})^2=(ba^{3})^4=[c,b]=[b,a^{3}]^2=1\rangle$& $ b, (bc)^{\pm 1}, ac$&$4$\\				
			%	$25$&$b^2, a^4, c^4$&$ c^2a^2, (ca^{3})^2, c^{3}bcb, (ba^{3})^4, (baba^{3})^2$ &$ b, bc, a^2bc, ac$&$4$\\
			$27$&$\langle a,b,c\mid a^2=b^2=c^2=[c,b]=(bc^a)^2=(ba)^4=(ca)^4=(acb)^4=1\rangle$& $ b, c, a$&$4$\\		
			%	$27$&$a^2, b^2, c^2$&$ (cb)^2, (baca)^2, (ba)^4, (ca)^4, c(acb)^3ab$ &$ b, c, a$&$4$\\
			$28$&$\langle a,b,c\mid a^2=c^2=b^4=[c,b]=(ab)^2=(ca)^4=1\rangle$&$ c, a, ab $&$8$\\		
			%	$28$&$ a^2, c^2, b^4$&$ cb^{3}cb, (ab)^2, (ca)^4 $ &$ c, a, ab $&$8$\\
			$29$&$\langle a,b,c\mid  c^2=a^4=b^4=[c,b]=(ca^{3})^4=ba^2b=a^{3}(a^{3})^b=[c,a^{3}]^2=1\rangle$& $c, a^{\pm 1}, (bc)^{\pm 1}$&$4$\\		
			%	$29$&$ c^2, a^4, b^4$&$a^{3}b^{3}a^{3}b, ba^2b, cb^{3}cb, (ca^{3})^4, (caca^{3})^2$ &$c, a, a^3, bc, a^2bc $&$4$\\
			$30$&$\langle a,b,c\mid  a^2=b^2=c^4=c^2(ab)^2=(ac)^4=[c,b]=1\rangle$&$ a, b, (bac)^{\pm 1}$&$4$\\		
			%	$30$&$ a^2, b^2, c^4$&$ c^3bcb, c^2(ab)^2, (ac)^4$ &$ a, b, bac, bca$&$4$\\
			$31$&$\langle a,b,c\mid a^2=b^4=c^4=(c^{2}a)^2=(b^{2}a)^2=[c,b]=[c^{3},a]b^{2}=(cab^3)^2=1\rangle$& $a, b^{\pm 1}, bab^2c$&$4$\\
			%	$31$&$a^2, b^4, c^4$&$ c^{3}b^{3}cb, (c^{2}a)^2, (b^{2}a)^2, cac^{3}ab^{2}, cacb^{3}ab^{3}$  &$a, b, b^3, bab^2c$&$4$\\
			$32$&$\langle a,b,c\mid a^4=b^4=c^4=c^2a^2=[c,b]=(a^{3})^ba^{3}=(ac)^2b^2=1\rangle$& $ a^{\pm 1}, (a^3b)^{\pm 1}, b^{\pm 1}, c^{\pm 1}, (ac)^{\pm 1} $&$4$\\
			%	$32$&$a^4,b^4,c^4$&$ c^2a^2, b^{3}a^{3}ba^{3}, c^{3}b^{3}cb, ab^{3}c^{3}a^{3}c^{3}b^{3} $ &$ a, a^{3}, b, a^3b, b^{3}a, b^{3}, c, c^{3}, ac, c^{3}a^{3} $&$4$\\
			$33$& $\langle a,b,c\mid a^2=c^4=b^4=[b,c^{3}]=b^{3}(c^2b)^a=c(b^{3}cb^{3})^a=1\rangle$&$a, (ac)^{\pm 1}, ac^2, c^{\pm 1}, b^{\pm 1}, (ab^3c^3)^{\pm 1}$	&$4$\\
			$34$& $\langle a,b,c\mid a^2=b^4=c^4=(ab)^2=(ac)^2=[c,b]=1\rangle$& $a, ab, ba, abaca$&$4$\\
			$37$&$\langle a,b,c\mid a^8=b^2=c^2=[c,a]=(cb)^2=(a^7b)^2a^6=1\rangle$ & $b, a^5b, a^7b, bc $&$4$\\
			$38$&$\langle a,b,c\mid a^8=b^2=c^2=[b,a]=[c,a]=(cb)^4=(cb)^2a^4=1\rangle$ &$ c, a^{\pm 1}, (ab)^{\pm 1}$&$4$\\
			$41$&$\langle a,b,c\mid a^4=b^4=c^2=b^2a^2=[c,a]=[c,b]=(a^{3}b^{3})^3ab^{3}=1\rangle$ &$  a^{\pm 1}, b^{\pm 1}, (a^3b)^{\pm 1}, (bc)^{\pm 1}$&$4$\\
			$42$&$\langle a,b,c\mid a^2=b^2=c^4=[c,a]=[c,b]=c^2(ab)^4=1\rangle$&$a, b, (ab)^2c $&$8$\\
			$43$&$\langle a,b,c\mid a^2=b^2=c^2=(cb)^2=(b^ac)^2=(ba)^3ba^c=1\rangle$ &$ b, bc, a $&$4$\\
			$44$&$\langle a,b,c\mid a^2=c^2=b^4=[c,b]=b^2(ca)^2=ab(ab^{3})^3=1\rangle$ &$c, a, b^{\pm 1}$&$4$\\
			$46$&$\langle a,b,c,d\mid a^2=b^2=d^2=c^2=(ca)^2=(da)^2=(cb)^2=(db)^2=(dc)^2=(ba)^4=1\rangle$ &$ babc, babd, abacd, babcd$&$4$\\
			$48$&$\langle a,b,c,d\mid  a^2=b^2=d^2=c^4=[d,c]=(da)^2=[c,a]=(db)^2=[c,b]=(cba)^2=1\rangle$ &$ a, b, abc, babd$&$4$\\
			$49$&$\langle a,b,c,d\mid a^2=b^2=d^2=c^2=(ca)^2=(db)^2=(dc)^2=(bca)^2=(da)^4=a^ba^d=1\rangle$&$c, d, ac, bad $&$4$\\	\hline			
			%												\bottomrule
		\end{tabular}
		\vspace{0.1cm}
		\caption{Some non-abelian groups $G=G_{32,r}$ of order $32$ which $\Gamma=\Cay(G,S)$ are not $2$-integral}\label{table:2}
		
	\end{table}
}

	{\footnotesize
	\begin{table}
		\centering
		\begin{tabular}{|c|c|c|c|}
			%	\toprule
			\hline
			$No.$&$G$&elements of $S$& $\deg(\Gamma)$ \\ \hline
			$1$& $G_{24,12}=\langle (1234),(12)\rangle\cong S_{4}$&$(34), (23), (1243)^{\pm 1}$&$4$\\
			$2$&$G_{24,13}=\langle a,b,c\mid a^3=b^2=c^2=[a,c]=[b,c]=1, aba=ba^2b\rangle\cong A_4\times\Bbb Z_2$&$ (ac)^{\pm 1}, b^ac, b$&$12$\\
			$3$&$G_{36,11}=\langle a,b,c\mid a^3=b^2=c^3=[a,c]=[b,c]=1, aba=ba^2b\rangle\cong A_4\times\Bbb Z_3$&$a^{\pm 1}, b, a^2bc, bac$&$4$\\
			$4$& $G_{48,50}=\langle b,c,d,e\rangle\rtimes\langle a\rangle\cong\Bbb Z_2^4\rtimes\Bbb Z_3; b^a=c, c^a=bc,d^a=e,e^a=de$&$b,a^{\pm 1},(ae)^{\pm 1}$&$4$\\
			$5$& $G_{48,3}=\langle a,b,c\mid a^3=b^4=c^4=[b,c]=1, b^a=c, c^a=(cb)^{-1}\rangle\cong \Bbb Z_4^2\rtimes\Bbb Z_3$&$ c^2,a^{\pm 1}, b^{\pm 1}$&$12$\\
			$6$& $G_{24,1}=\langle a,b\mid a^8=b^3=1, b^a=b^2\rangle\cong\Bbb Z_3\rtimes\Bbb Z_8$&$(a^2b)^{\pm 1},a^{\pm 1}$&$4$\\
			%	$7$&$G_{50,3}=\langle a,b,c\mid a^2=b^5=c^5=(ac)^2=[b,a]=[b,c]=1\rangle\cong\Bbb Z_5\times D_{10}$&$ab,ab^2c,ab^4,ab^3c$&$8$\\
			$7$&$G_{50,4}=\langle a,b,c\mid a^2=b^5=c^5=[b,c]=1, b^a=b^{-1},c^a=c^{-1}\rangle\cong\Bbb Z_5^2\rtimes\Bbb Z_2$&$a,ab,ac$&$8$\\
			$8$& $G_{80,49}=\langle a,b\mid a^5=b^2=(bb^a)^2=(ba^{-1})^5=(ba^{-2}ba^2)^2=1 \rangle=\Bbb Z_2^4\rtimes\Bbb Z_5$&$ a^{\pm 1}, (a^{-1}b)^{\pm 1} $&$4$\\
			$9$&$G_{18,3}=\langle a,b,c\mid a^2=b^3=c^3=(ab)^2=1,[a,c]=[b,c]=1\rangle\cong S_3\times\Bbb Z_3$&$a, (cb)^{\pm 1}, (acb)^{\pm 1}$ &$4$\\ 
			%$10$&$G_{24,1}=\langle a,b\mid a^3=b^8=1,a^b=a^2\rangle\cong\Bbb Z_3\rtimes\Bbb Z_8$&$b, b^2a, b^{-1}, b^6a^2$&$4$\\ 
			$10$&$G_{24,5}=\langle a,b,c\mid a^2=b^3=(ab)^2=c^4=1, [a,c]=[b,c]=1\rangle\cong S_3\times\Bbb Z_4$& $a, (acb)^{\pm 1} $&4\\ 
			$11$&$G_{24,8}=\langle a,b,c\mid  a^2=b^2=c^3=(ac)^2=b^cb=(ba)^4=1\rangle\cong (\Bbb Z_6\times\Bbb Z_2)\rtimes\Bbb Z_2$&$ a, abc, (ab)^3c $&$4$\\ 
			$12$&$G_{24,10}=\langle a,b,c\mid a^2=b^2=(ab)^4=c^3=1,[c,a]=[c,b]=1\rangle\cong  D_8\times \Bbb Z_3 $ & $a, (bc)^{\pm 1} $&$4$\\ 
			$13$&$G_{24,14}=\langle a,b,c,d\mid a^2=b^2=c^2=d^3=(cd)^2=1\rangle\cong \Bbb Z_2^2\times S_3; a,b\in Z(G)$&$ c, ca, cbd$&$4$\\
			%$\langle a,b,c,d \mid a^2=b^2=(ab)^2=c^2=d^3=(cd)^2=[a,c]=[a,d]=[b,c]=[b,d]=1\rangle\cong \Bbb Z_2^2\times S_3$&$ c, ca, cbd$&$4$\\ 
			$14$&$G_{48,11}=\langle a,b,c\mid a^4=b^3=c^4=1, b^c=b^2, [a,b]=[a,c]=1\rangle$&$ c^{\pm 1}, (ab)^{\pm 1}$&$4$\\ %[ 48, 11 ]
			$15$&$G_{48,12}=\langle a,b,c\mid a^4=b^4=c^3=1, [b,c]=1,c^a=c^2,a^b=a^3\rangle$&$a^{\pm 1}, (bc)^{\pm 1}$&$4$\\ %[48,12]
			$16$&$G_{48,13}=\langle a,b\mid a^{12}=b^4=1, a^b=a^{-1}\rangle\cong\Bbb Z_{12}\rtimes\Bbb Z_4$&$  b^{\pm 1}, (ba)^{\pm 1} $&$4$\\ %[48,13]
			$17$&$G_{48,19}=\langle a,b,c\mid a^6=b^2=c^4=[a,b]=1, a^c=a^5b,b^c=b\rangle$&$c^{\pm 1},  a^{\pm 1}$&$4$\\ %[48,19]
			$18$&$G_{48,21}=\langle a,b,c\mid  b^2=c^3=a^4=(ba)^4=[c,a]=[c,b]=[a^2,b]=1\rangle$&$ a^{\pm 1}, (bc)^{\pm 1}$&$4$\\ %[48,21]
			$19$&$G_{48,22}=\langle a,b,c\mid a^4=b^4=c^3=1, b^a=b^{-1},[c,a]=[c,b]=1\rangle$&$ a^{\pm 1}, (abc)^{\pm 1} $&$8$\\ %[48,22]
			$20$&$G_{48,34}=\langle a,b,c,d\mid a^4=b^2=d^3=c^2a^2=cc^{a^{3}}=d^ad=[d,c]=1\rangle; b\in Z(G) $&$   (ac^{3})^{\pm 1}, a^{\pm 1}, (ab)^{\pm 1}, (ad)^{\pm 1} $&$4$\\ %[48,34]
			$21$&$G_{48,42}=\langle a,b,c,d\mid a^4=b^2=c^2=d^3=1, d^a=d^2\rangle; b,c\in Z(G) $&$ (ab)^{\pm 1}, a^{\pm 1}, (acd)^{\pm 1}$&$4$\\ %[48,42]
			$22$&$G_{48,46}=\langle a,b,c\mid a^4=c^6=[c,a]=[c,b]=1, a^2=b^2,a^b=a^{-1}\rangle\cong\Bbb Z_6\times Q_8$&$(ac^4)^{\pm 1}, (bc^4)^{\pm 1}, (ac)^{\pm 1}$&$4$\\%[ 48, 46 ]
			$23$&$D_{16}=\langle a,b\mid a^8=b^2=(ab)^2=1\rangle$&$b,ba$&$4$\\
			$24$&$QD_{16}=\langle a,b\mid a^8=b^2=1, bab=a^3\rangle$&$(ab^{-1})^{\pm 1}, b, a^{\pm 1}$&$4$\\
			$25$&$D_{18}=\langle a,b\mid a^9=b^2=(ab)^2=1\rangle$&$b, a^{\pm 1}$&$3$\\
			$26$&$D_{24}=\langle a,b\mid a^{12}=b^2=(ab)^2=1\rangle$&$b, ba$&$4$\\
			$27$&$G_{60,5}=\langle (12345), (345)\rangle\cong A_5$&$(345)^{\pm 1}, (13)(24)$&$4$\\%5.1
			$28$&$G_{27,3}=\langle a,b\mid a^3=b^3=c^3=1,c=[a,b], [a,c]=[b,c]=1\rangle$&$a^{\pm 1},b^{\pm 1},(ab)^{\pm 1}$&$3$\\%5.2
			$29$&$G_{125,3}=\langle a,b\mid a^5=b^5=c^5=1,c=[a,b], [a,c]=[b,c]=1\rangle$&$a^{\pm 1},b^{\pm 1},(ab)^{\pm 1}$&$64$\\%5.2
		%	$$&$G_{405,15}=\langle b_1,b_2,b_3,b_4\rangle\rtimes\langle a\rangle\cong \Bbb Z_3^4\rtimes\Bbb Z_5;b_1^a=b_2b_3^2b_4^2,~b_2^a=b_1^2b_3,~b_3^a=b_2^2b_4,~b_4^a=b_1b_3^2b_4^2$&$a^{\pm 1},b_1^{\pm 1}$&$$\\%5.3
		%	$$&$G=\langle a,b_1,b_2\rangle$&$a^{\pm 1},b_1^{\pm 1}$&$$\\%5.3
			$30$&$G_{50,4}=\langle a,b,c\mid a^2=b^5=c^5=(ab)^2=(ac)^2=[b,c]=1\rangle\cong\Bbb Z_5^2\rtimes\Bbb Z_2$&$a,ab,ac$&$8$\\%5.7
			$31$&$G_{36,9}=\langle a,b\mid a^4=b^3=(a^{-1}ba^{-1})^2=a^{-1}ba^{-1}(b^{-1}a)^2b^{-1}=1\rangle$&$a^{\pm 1},a^2b$&$12$\\%5.10
			$32$&$G_{36,13}=\langle a,b,c,d\mid a^2=b^2=c^3=d^3=(ac)^2=(ac)^2=(ba)^2=b^cb= b^db=d^{-1}d^c=1\rangle$&$a,ac,abd$&$4$\\%5.10
			$33$&$G_{36,6}=\langle  a,b,c\mid b^3=c^3=a^4=c^ac=[b,a]=[c,b]=1\rangle$&$ a^{\pm 1}, (abc)^{\pm 1}$&$4$\\%5.11
			$34$&$D_{20}=\langle a,b\mid a^{10}=b^2=(ab)^2=1\rangle$&$b,ba$&$4$\\
			$35$&$G_{40,1}=\langle a,b\mid a^8=b^5=1, b^a=b^{-1}\rangle\cong\Bbb Z_5\rtimes\Bbb Z_8$&$a^{\pm 1}, b^{\pm 1}$&$4$\\
			$36$&$G_{40,3}=\langle a,b\mid a^8=b^5=1, b^a=b^{2}\rangle\cong\Bbb Z_5\rtimes\Bbb Z_8$&$a^{\pm 1}, b^{\pm 1}$&$16$\\
			$37$&$G_{40,7}=\langle a,b,c\mid a^4=b^5=c^2=[a,c]=[b,c]=1, b^a=b^{-1}\rangle\cong\Bbb Z_2\times Dic_{20}$&$a^{\pm}, (acb)^{\pm}$&$4$\\
			  \hline

			%$\langle a,b,c,d\mid a^4=b^2=c^3=d^2=[c,a]=[c,b]=[c,d]=[a,d]=1, a^b=ad,d^b=a^4d \rangle\cong\Bbb Z_3\times((\Bbb Z_4\times\Bbb Z_2)\rtimes\Bbb Z_2) $&$\{a, a^{-1}, bc, (bc)^{-1}\}$&$4$\\ %[48,21]
			
			% $\langle a,b,c,d\mid a^4=b^2=d^3=[c,d]=[b,a]=[b,c]=[b,d]=1, a^2=c^2, a^c=a^{-1},d^a=d^2\rangle\cong\Bbb Z_2\times(\Bbb Z_3\rtimes Q_8) $&$\{a, b, ab, a^3, ba^3, abcd, (abcd)^{-1} \}$&$4$\\ %[48,34]
			
			% $14$&$\langle a,b,c,d\mid a^4=b^2=c^2=d^3=[b,c]=[b,a]=[b,d]=[c,a]=[c,d]=1, d^a=d^2\rangle\cong (\Bbb Z_2\times\Bbb Z_2)\times(\Bbb Z_3\rtimes\Bbb Z_4) $&$\{ a, ab, a^3, a^3b, acd, a^3cd\}$&$4$\\ %[48,42]
			%												\bottomrule
		\end{tabular}
		\vspace{0.1cm}
		\caption{Some Cayley graphs $\Gamma=\Cay(G,S)$ which are not $2$-integral.}\label{Table23}
		
	\end{table}
}

\section{Finite non-abelian Cayley $2$-integral groups}

\begin{lem}\label{solv}
	Let $G$ be a finite Cayley $2$-integral group. Then $G$ is solvable.
\end{lem}
\begin{proof}
By Corollary \ref{2,3,5}, \( G \) is a \( \{2,3,5\} \)-group. According to Burnside's \( pq \)-Theorem, we can assume that the order of \( G \) has three distinct prime divisors.  Since Cayley \( 2 \)-integrality is preserved through quotients, by induction on $|G|$, it suffices to demonstrate that \( G \) possesses a non-trivial solvable normal subgroup. 
Let \( N \) denote a minimal normal subgroup of \( G \). As a result, \( N \) can be expressed as a direct product of isomorphic simple groups. We can write \( N \cong T^k \), where \( T \) is a finite simple group and \( k \geq 1 \). If \( N \) is non-abelian, it follows that \( T \) must be one of the groups: \( A_5, A_6, L_2(7), L_2(8), L_2(17), L_3(3), U_3(3) \), or \( U_4(2) \) \cite[p.383]{Herzog}. Given that \( T \) is a \( \{2,3,5\} \)-group, we conclude that \( T \) can only be \( A_5, A_6, \) or \( U_4(2) \). 
It is important to note that \( A_5 \) is present as a subgroup in all these simple groups. By Table \ref{Table23}(27), $A_5$ is not a Cayley $2$-integral group. Since Cayley \( 2 \)-integrality is invariant under the formation of subgroups, we reach our conclusion.
\end{proof}

\begin{lem}\label{elab}
A finite $p$-group, $p=3,5$,  is Cayley $2$-integral if and only if it is elementary abelian.
\end{lem}
\begin{proof}
One direction is established by Proposition \ref{baseab}. Conversely, let \( G \) represents a Cayley \( 2 \)-integral group. According to Lemma \ref{2,3,5}, \( G \) has exponent \( p \). It suffices to demonstrate that \( G \) is abelian. Assume, for the sake of contradiction, that \( G \) is a non-abelian \( p \)-group of minimal possible order that is also a Cayley \( 2 \)-integral group. Under this assumption, every proper subgroup of \( G \) must be abelian since every subgroup of a Cayley \( 2 \)-integral group is Cayley \( 2 \)-integral as well. Consequently, \( G \) is a minimal non-abelian \( p \)-group, which implies that it has nilpotent class \( 2 \) and is \( 2 \)-generated. Given that \( G \) has exponent \( p \), its order is \( p^3 \), and \( G/Z(G) \cong \mathbb{Z}_p \times \mathbb{Z}_p \) while \( Z(G) \) is cyclic. Hence, by \cite[\textsection 1, Exercise 8.a]{Berkovich1}, 
$G=\langle a,b\mid a^p=b^p=c^p=1,c=[a,b], [a,c]=[b,c]=1\rangle$. Then we get a contradiction by rows $(28)$ and $(29)$ of Table \ref{Table23}.
\end{proof}

%\begin{core}
%	Let $G$ be a $\{3,5\}$-group which is a Cayley $2$-integral, then one of the following holds:
%	\begin{itemize}
%		\item [(a)] $G$ is elementary abelian
%		\item [(b)] $G$ is a group with a normal Sylow $p$-subgroup and a Sylow $q$-subgroup of order $q$ such that $p\neq q$ and $p,q\in\{3,5\}$.
%	\end{itemize}
%\end{core}
\begin{lem}\label{3,5}
	A finite $\{3,5\}$-group $G$ is Cayley $2$-integral if and only if it is an elementary abelian $p$-group for $p=3$ or $5$.
\end{lem}
\begin{proof}
One direction is clear from Lemma \ref{elab}. Conversely, suppose, for the sake of contradiction, that a finite $\{3,5\}$-group $G$ is Cayley $2$-integral but not elementary abelian. According to Lemma \ref{elab}, it follows that both $3$ and $5$ divide $|G|$. Additionally, by Corollary \ref{order}, we conclude that $G$ is non-abelian. Applying Burnside's $pq$-Theorem indicates that $G$ is solvable. Consequently, $G$ must have a non-trivial elementary abelian normal subgroup $N$. Moreover, $C:=C_G(N)$ is also a normal subgroup of $G$. As established by Corollary \ref{order}, since $G$ contains no element of order $15$, it follows that $C$ is a $p$-group, where $p=3$ or $5$. Let $P$ denotes the Sylow $p$-subgroup of $G$. Thus, we have $N\leq C\leq P$. On the other hand, since $P$ is elementary abelian according to Lemma \ref{elab}, we have $P\leq C$. Therefore, $P=C$, making it a normal Sylow $p$-subgroup of $G$. Let $Q$ represent a Sylow $q$-subgroup of $G$, with the condition that $q\neq p$. The subgroup $Q$ acts Frobeniusly on $P$, as $G$ contains no element of order $15$. Given that $Q$ has odd order, it must be cyclic, and thus of order $q$ by Lemma \ref{elab} and \cite[10.5.5]{Robinson}. Let $Q=\langle x\rangle$ and let $1\neq y\in P$. Consequently, $y,y^x,\ldots,y^{x^{q-1}}$ are distinct, since $Q$ acts fixed-point freely on $P$. We define $P_0:=\langle y,y^x,\ldots,y^{x^{q-1}}\rangle$. Since $P$ is elementary abelian, $P_0$ forms a group of order $p^k$, with $1\leq k\leq q$. Additionally, $H:=P_0Q$ is a subgroup of $G$ with order $p^kq$. Given that $P_0\unlhd H$, $H$ is a Frobenius group with Frobenius complement $Q$ and Frobenius kernel $P_0$. Thus, either $|H|=3^k5$ where $1\leq k\leq 5$, or $|H|=5^k3$ where $1\leq k\leq 3$. In the former case, since $5\mid 3^k-1$, we conclude $k=4$. In the latter case, $3\mid 5^k-1$, which leads us to $k=2$, as noted in \cite[Exercise 8.5.6]{Robinson}. Hence, we find $|H|=3^4.5$ or $5^2.3$. Now we examine all possibilities by GAP. There are eleven non-abelian groups of order $3^4.5$. Among these, the only group that has no element of order $15$ is $G=\langle a,b_1,b_2,b_3,b_4\rangle$, where $a^5=b_1^3=b_2^3=b_3^3=b_4^3=1$, $[b_i,b_j]=1$ for all $i,j$, and 
\[a^{-1}b_1a=b_2b_3^2b_4^2,~a^{-1}b_2a=b_1^2b_3,~a^{-1}b_3a=b_2^2b_4,~a^{-1}b_4a=b_1b_3^2b_4^2.\]
 Put $S=\{a,b_1,a^{-1},b_1^{-1}\}$ and $\Gamma=\Cay(G,S)$. Then one can check by GAP that the characteristic polynomial of $\Gamma$ is $f(x)g(x)$, where  
 \begin{eqnarray*}
 	&&f(x)=(x-4)(x-1)^{20}(x+2)^{30}(x^2-3x-2)^{10}(x^2-3x+1)^{32}(x^2-5)^{20}(x^2-2)^{20},\\
 	&&g(x)=(x^2+3x+1)^{50}(x^3-4x^2-2x+14)^{10}(x^3-x^2-8x+2)^{10}(x^3-x^2-5x-1)^{10}.
 \end{eqnarray*}
 which implies that $G$ is not Cayley $2$-integral, a contradiction. Furthermore, the only non-abelian group of order $5^2.3$ is $G=\langle a,b_1,b_2\rangle$, where $[b_1,b_2]=1$, $a^3=b_1^5=b_2^5=1$, $a^{-1}b_1a=b_1b_2^2$ and $a^{-1}b_2a=b_1b_2^3$. Put $S=\{a,b_1,a^{-1},b_1^{-1}\}$ and $\Gamma=\Cay(G,S)$. Then the characteristic polynomial of $\Gamma$ is 
 \begin{eqnarray*}
 	(x-4)(x-1)^2(x^2-3x+1)^6(x^2-2x-4)^6(x^2+3x+1)^{12}(x^4-7x^2+11)^6.
 \end{eqnarray*}
  Thus $G$ is not a Cayley $2$-integral group. This contradiction completes our assumption, thus concluding the proof.
\end{proof}

\begin{core}\label{key}
	Let $G$ be a finite Cayley $2$-integral group. Then $G$ is a $\{2,3\}$-group or $\{2,5\}$-group.
\end{core}
\begin{proof}
By Lemma \ref{solv}, $G$ is solvable. Thus $G$ has a $\{3,5\}$-Hall subgroup $H$ and by Lemma \ref{2,3,5}, $|G:H|=2^n$ for some integer $n$. Now Lemma \ref{3,5} implies that $G$ is a $\{2,3\}$-group or $\{2,5\}$-group.
\end{proof}

\subsection{Non-abelian Cayley $2$-integral $2$-groups}

By Corollary \ref{key}, a finite Cayley $2$-integral group is a $\{2,p\}$-group, where $p=3$ or $5$. In this section, we  classify all finite non-abelian Cayley $2$-integral $2$-groups. We start with the following result.

\begin{lem}\label{16}
	A finite non-abelian group $G$ of order $16$ is Cayley $2$-integral if and only if $G\ncong D_{16}, QD_{16}$.
\end{lem}
\begin{proof}
	There are nine non-abelian groups of order $16$ including the dihedral group $D_{16}$ and the semidihedral group $QD_{16}$, which both are not Cayley $2$-integral by Table \ref{Table23}. The remaining groups are Cayley $2$-integral by Corollary \ref{coremain}.
\end{proof}

	\begin{lem}\label{32}
		A finite non-abelian $G$ of order $32$ is Cayley $2$-integral group if and only if $G\cong  G_{32,50}, G_{32,35}, G_{32,26}, G_{32,23}$ or $G_{32,47}\cong Q_8\times\Bbb Z_2^2$.
	\end{lem}
	\begin{proof}
		One direction is clear by Lemma \ref{32yes} and Theorem \ref{azhv}. Conversely, suppose that $G$ is a non-abelian Cayley $2$-integral group of order $32$. Since $G$ is non-abelian, by GAP, we have $G=G_{32,r}$, where $r\in\{2,\ldots,50\}$ and $r\neq 3,16,21,36,45$. By Corollary \ref{2,3,5}, $r\neq 17,18,19,20$, since the exponent of  groups $G_{32,17},G_{32,18},G_{32,19}, G_{32,20}$  is $16$. Also, by Lemma \ref{16}, $r\neq 39, 40$, because $G_{39,40}$ has a subgroup isomorphic to $D_{16}$ and $G_{32,40}$ has a subgroup isomorphic to $QD_{16}$. Now by Table \ref{table:2}, the only possibilities are $r=23,26,35,47,50$, which completes the proof.
	\end{proof}
	
	Similarly, one can prove the following result.
	\begin{lem}\label{68}
		A finite non-abelian group $G$ of order $2^6$ is Cayley $2$-integral if and only if $G\cong Q_8\times\Bbb Z_2^{3}$.
	\end{lem}
\begin{lem}\label{Dedekind}
	A finite non-abelian $2$-group $G$ of order $2^n\geq 128$ is Cayley $2$-integral if and only if $G\cong Q_8\times\Bbb Z_2^{n-3}$.
\end{lem}
\begin{proof}
Since \( Q_8 \times \mathbb{Z}_2^{n-3} \) is  Cayley integral by Theorem \ref{azhv}, one direction of our argument is straightforward. Now, let's consider the opposite scenario where \( G \) is  Cayley 2-integral. We will apply induction on the order of \( G \). Let \( x \) represent a central involution, and consider the quotient group \( \overline{G} = G/\langle x \rangle \). According to Corollary \ref{68}, we establish that \( |\overline{G}| \geq 64 \), and by the induction hypothesis, it follows that \( \overline{G} \cong Q_8 \times \mathbb{Z}_2^{n-4} \), where \( n-4 \geq 3 \). 
As a result, \( \overline{G} \) possesses two distinct maximal subgroups \( \overline{M_i} \) for \( i=1,2 \), both of which are isomorphic to \( Q_8 \times \mathbb{Z}_2^{n-5} \). Employing induction, we find that \( M_i \cong Q_8 \times \mathbb{Z}_2^{n-4} \) for \( i=1,2 \). This implies that \( \langle x \rangle \) is a direct factor in both \( M_i \).
Furthermore, we can express \( M_i \) as \( (M_1 \cap M_2) \times \langle y_i \rangle \) for some involution \( y_i \) within \( M_i \), where \( i=1,2 \). Here, \( M_1 \cap M_2 = N \times \langle x \rangle \) for some subgroup \( N \cong Q_8 \times \mathbb{Z}_2^{n-6} \). Notably, we have \( G = M_1M_2 = (M_1)\langle y_2 \rangle \cong ((N \times \langle y_1 \rangle)\langle y_2 \rangle)\times \langle x \rangle \). Since \( G/\langle x \rangle \cong Q_8 \times \mathbb{Z}_2^{n-4} \cong (N \times \langle y_1 \rangle)\langle y_2 \rangle \), we deduce that \( (N \times \langle y_1 \rangle)\langle y_2 \rangle \cong N \times \langle y_1 \rangle \times \langle y_2 \rangle \). 
Therefore, we conclude that \( G \cong N \times \langle y_1 \rangle \times \langle y_2 \rangle \times \langle x \rangle \cong Q_8 \times \mathbb{Z}_2^{n-3} \), as was to be shown. 
\end{proof}

By Lemma \ref{dihedral}, $D_8$ is Cayley $2$-integral. Also $Q_8$ is Cayley integral and so it is Cayley $2$-integral. Now combining Lemmas \ref{16}-\ref{Dedekind}, we give a complete classification of non-abelian Cayley $2$-integral $2$-groups as follows:
\begin{pro}\label{non-ab-2groups}
	Let $G$ be a finite non-abelian $2$-group of order $2^n$. Then $G$ is a Cayley $2$-integral group if and only if
	\begin{itemize}
		\item[(1)] $n=3$, $G\cong D_8$ or $Q_8$.
		\item[(2)] $n=4$, $G\ncong D_{16}, QD_{16}$.
		\item[(3)] $n=5$, $G\cong  G_{32,50}, G_{32,35}, G_{32,26}, G_{32,23}$ or $Q_8\times\Bbb Z_2^2$.
		\item[(4)] $n\geq 6$, $G\cong Q_8\times\Bbb Z_2^{n-3}$. 
	\end{itemize}
\end{pro}

\subsection{Non-abelian Cayley $2$-integral $\{2,3\}$-groups}

\begin{lem}\label{nilpotent}
  Let $G$ be a finite non-abelian nilpotent $\{2,3\}$-group which is not a $p$-group. Then $G$ is Cayley $2$-integral if and only if $G\cong  Q_8\times\Bbb Z_3$.
	\end{lem}
	\begin{proof}
		One direction follows from Lemma \ref{q8z3}. Suppose, towards a contradiction, that  $G\ncong \Bbb Z_3\times Q_8$ is a Cayley $2$-integral. Then, by Lemma \ref{sbg-qoutient}, $G=P\times Q$, where $P$ is a group of order $2^\alpha$ given in Theorem \ref{non-ab-2groups} and $Q\cong\Bbb Z_3^\beta$ by Lemma \ref{3,5}, for some positive integers $\alpha,\beta$.		
 Let $1\neq x\in Q$ then $H:=\langle P, x\rangle=P\times\langle x\rangle$ is Cayley $2$-integral, by Lemma \ref{sbg-qoutient}. We know $\chi_j:x^k\mapsto \omega^{jk}$, $j=0,1,2$, where $\omega=e^{\frac{2\pi\textbf{i}}{3}}$ are all inequivalent irreducible representations of $\langle x\rangle$.  By Theorem \ref{non-ab-2groups}, we deal with the following cases:
		
		\textbf{Case 1.} $P\cong L$ or $L\times\Bbb Z_2 $, where $L=\langle a,b\mid a^4=b^2=(ab)^2=1\rangle\cong D_8$.  Then 
		\[\rho:a^k\mapsto \begin{bmatrix}
			\textbf{i}&0\\
			0&-\textbf{i}
	\end{bmatrix},~~ba^k\mapsto\begin{bmatrix}
		0&-\textbf{i}\\
		\textbf{i}&0
		\end{bmatrix},~~k=0,1,2,3\]
		is an irreducible representation of $L$. Put $S=\{b,bax,bax^2\}$. Then the characteristic polynomial of $(\rho\otimes\chi_0)(S)$ and $(\rho\otimes\chi_1)(S)$, are $\lambda^2-5$ and $\lambda^2-2$, respectively. Hence $\sqrt{2}$ and $\sqrt{5}$ are eigenvalues of $\Cay(L\times\langle x\rangle,S)$ which means $H$ is not Cayley $2$-integral, a contradiction.
		
		\textbf{Case 2.}  $P\cong\Bbb Z_8\rtimes\Bbb Z_2$ or $Q_{16}$. Then $P$ has an element of order $8$ and so $H$ has  an element of order $24$, which is impossible by Corollary \ref{order}.
		
			\textbf{Case 3.} $P\cong L\times  K^l$, $l\geq 1$, where $L=\langle a,b\mid a^4=1, a^2=b^2, a^b=a^{-1}\rangle\cong Q_8$ and $K=\langle c\rangle\cong\Bbb Z_2$. Clearly, $\mu_j:c^k\mapsto (-1)^{kj}$, $j,k=0,1$ are irreducible representations of $K$, and by the notations of the proof of Lemma \ref{q8z3}, $\rho_5$ is an irreducible representation of $L$.  Let $S=\{ax,acx,bx,a^{-1}x^{-1},a^{-1}cx^{-1},b^{-1}x^{-1}\}$. Then the characteristic polynomial of $(\rho_5\otimes\mu_0\otimes\chi_2)(S)$ and $(\rho_5\otimes\mu_1\otimes\chi_2)(S)$ are $\lambda^2-15$ and $\lambda^2-3$, respectively. Hence $\sqrt{15}$ and $\sqrt{3}$ are eigenvalues of $\Cay(L\times K\times\langle x\rangle,S)$, which implies that $H$ is not Cayley $2$-integral. 
		
		\textbf{Case 4.} $P=G_{16,3}=\langle a,b\mid a^4=b^2=1, (ab)^4=1, ba^2=a^2b\rangle \cong (\Bbb Z_4\times\Bbb Z_2)\rtimes\Bbb Z_2$. Then 
		\begin{eqnarray*}
			\rho:a\mapsto\begin{bmatrix}
				0&1\\
				1&0
			\end{bmatrix},~~b\mapsto\begin{bmatrix}
				1&0\\
				0&-1
			\end{bmatrix}
		\end{eqnarray*}
		is an irreducible representations of $P$. Let $S=\{a,bx,bx^{-1},a^{-1}\}$. Now the characteristic polynomial of $(\rho\otimes\chi_0)(S)$, and $(\rho\otimes\chi_1)(S)$ are $\lambda^2-8$ and $\lambda^2-5$, respectively. This means $H$ is not Cayley $2$-integral.
		
		\textbf{Case 5.} $P=G_{16,6}=\langle a,b\mid a^8=b^2=1, bab=a^5\rangle\cong \Bbb Z_4\rtimes\Bbb Z_4$. Then 
		\begin{eqnarray*}
			\rho:a\mapsto\begin{bmatrix}
				0&\textbf{i}\\
				1&0
			\end{bmatrix},~~b\mapsto\begin{bmatrix}
			1&0\\
			0&-1
			\end{bmatrix}
		\end{eqnarray*}
		is an irreducible representation of $P$. Let 	$S=\{a,bx,bx^{-1},a^{-1}\}$. Now the characteristic polynomial of $(\rho\otimes\chi_0)(S)$, and $(\rho\otimes\chi_1)(S)$ are $\lambda^2-6$ and $\lambda^2-3$, respectively. This means that $H$ is not Cayley $2$-integral.

		\textbf{Case 6.} $P\cong G_{32,50},G_{32,35}$ or  $G_{32,26}$. Then $P$ has a subgroup isomorphic to $Q_8\times\Bbb Z_2$. Hence  $H$ is not Cayley $2$-integral by Case $3$.
		
		\textbf{Case 7.} $P\cong G_{32,23}$. Then $P$ has a subgroup isomorphic to $\Bbb Z_4\rtimes\Bbb Z_4$ and so $H$ is not Cayley $2$-integral by Case 5.
	\end{proof}
	
	The following result of H. Fitting is well-known \cite[Theorem 4.34]{Issacs}, but we recall it for convenience of the reader.
	\begin{lem}(Fitting's Lemma)\label{fit}
		Let the group $A$ acts via automorphisms on an abelian group $G$, and assume that $A$ and $G$ are finite and that $(|A|,|G|)=1$. Then $G=C_G(A)\times [G,A]$. 
	\end{lem}
	\begin{lem}\label{main3k2}
		Let $G$ be a finite non-abelian group of order $2p^k$, $k\geq 2$, $p\neq 2$ be a prime, and $P\cong\Bbb Z_p^k$ be a Sylow $p$-subgroup of $G$. Put
		\begin{itemize}
			\item[(a)] $H_1:=\langle a,b,c\mid a^2=b^p=c^p=(ac)^2=[a,b]=[b,c]=1\rangle\cong D_{2p}\times\Bbb Z_p$,
			\item[(b)] $H_2:=\langle a,b,c\mid a^2=b^p=c^p=(ab)^2=(ac)^2=[b,c]=1\rangle\cong\Bbb Z_p^2\rtimes\Bbb Z_2$,
			\item[(c)] $H_3:=\langle a,b,c,d\mid a^2=b^p=c^p=d^p=(ab)^2=(ac)^2=(ad)^2=[b,c]=[b,d]=[c,d]=1\rangle\cong \Bbb Z_p^3\rtimes\Bbb Z_2$.
		\end{itemize} 
			If $k=2$ then either $G\cong H_1$ or $H_2$, and otherwise, either $H_1\leq G$ or $H_3\leq G$. Furthermore, $H_1$ and $H_3$ are not Cayley $2$-integral. In particular, if $G$ is Cayley $2$-integral, then $p=3$ and $G\cong H_2$.
	\end{lem}
	\begin{proof}	 
	Let $P$ be a Sylow $p$-subgroup of $G$. As $|G:P|=2$, $P$ is normal in $G$ and by Lemma \ref{3,5} $P$ must be abelian. Now by using Fitting's Lemma we get that $G= P_1\times (P_2\rtimes Q)$ where $Q\cong \Bbb Z_2$, $P_1=C_P(Q)$  and $P_2=[P, Q]$. So if $k=2$ we have that either $P_1=1$ or $P_1\cong \Bbb Z_3$, implying that $G\cong H_1$ or $G\cong H_2$. Moreover, if $k>2$, either $H_1\leq G$ or $H_3\leq G$.

%$\langle x_1,yx_1\rangle=\langle x_1,(yx_1)^p,(yx_1)^2\rangle=\langle  x_1,(yx_1)^p\rangle\times\langle (yx_1)^2\rangle \cong D_{2p}\times\Bbb Z_p$. Hence, we may assume that $(yx_1)^2=1$ and so $\langle x_1,y\rangle\cong D_{2p}$. If $[x_i,y]=1$ for some $2\leq i\leq k$, then $\langle x_1,y,x_i\rangle=\langle x_1,y\rangle\times\langle x_i\rangle\cong D_{2p}\times\Bbb Z_p$, which means that we may assume that for all $1\leq i\leq k$, $[x_i,y]\neq 1$. Furthermore, by the above argument, we may assume that $(yx_i)^2=1$, for all $1\leq i\leq k$. If $k=2$ then $G\cong H_2$ and if $k\geq 3$ then
%$\langle y,x_1,x_2,x_3\rangle\cong H_3$ is a subgroup of $G$. This proves the first part.

 By considering the irreducible representations of $D_{2p}$ and $\Bbb Z_p$, it is easy to see that the following maps are two inequivalent irreducible representations of $H_1$:
\begin{eqnarray*}
	&&\phi: a\mapsto \begin{bmatrix}
		0&1\\
		1&0
	\end{bmatrix},~~b\mapsto \begin{bmatrix}
		1&0\\
		0&1
	\end{bmatrix},~~c\mapsto\begin{bmatrix}
		\omega&0\\
		0&\omega^{-1}
	\end{bmatrix},\\
	&&\psi: a\mapsto \begin{bmatrix}
		0&1\\
		1&0
	\end{bmatrix},~~b\mapsto \begin{bmatrix}
		\omega&0\\
		0&\omega
	\end{bmatrix},~~c\mapsto\begin{bmatrix}
		\omega&0\\
		0&\omega^{-1}
	\end{bmatrix},
\end{eqnarray*}
where $\omega=e^{\frac{2\pi\textbf{i}}{p}}$. Let $S=\{ a, ab, ac, bc, ab^{-1}, b^{-1}c^{-1} \}$. Then the eigenvalues of $\phi(S)$ and $\psi(S)$ are $t\pm\sqrt{10+3t}$ and $\frac{t^2\pm\sqrt{t^4+12t+24}}{2}$, respectively, where $t=2\cos(2\pi/p)$. This proves that $H_1$ is not Cayley $2$-integral.

Now we consider the group $H_3$. Then
%\[H_3=\langle a,b,c,d\mid a^2=b^3=c^3=d^3=(ab)^2=(ac)^2=(ad)^2=[b,c]=[b,d]=[c,d]=1\rangle.\]
%Then
\begin{eqnarray*}
	\rho_1: b^{n_1}c^{n_2}d^{n_3}a^m\mapsto \begin{bmatrix}
		\alpha^{n_1+n_2}&0\\
		0&\alpha^{-(n_1+n_2)}
	\end{bmatrix}\begin{bmatrix}
		0&1\\
		1&0
	\end{bmatrix}^m
\end{eqnarray*}
and 
\begin{eqnarray*}
	\rho_2: b^{n_1}c^{n_2}d^{n_3}a^m\mapsto \begin{bmatrix}
		\alpha^{n_2+n_3}&0\\
		0&\alpha^{-(n_2+n_3)}
	\end{bmatrix}\begin{bmatrix}
		0&1\\
		1&0
	\end{bmatrix}^m,
\end{eqnarray*}
where $0\leq n_1,n_2,n_3\leq p-1$, $0\leq m\leq 1$ and $\alpha=e^{\frac{2\pi\textbf{i}}{p}}$, are irreducible representations of $H_3$. Put $S=\{a,b^{-1}a,c^{-1}a,b^{-1}cd^{-1}a,c^{-1}da\}$. Then the eigenvalues of  $\rho_1(S)$ and $\rho_2(S)$ are $\pm\sqrt{13+12\cos(2\pi/p)}$ and $\pm\sqrt{17+8\cos(2\pi/p)}$, respectively. This means that $H_3$ is not Cayley $2$-integral.

Now suppose that $G$ is Cayley $2$-integral. Then $p\in\{3,5\}$ and $G\cong H_2$. By Table \ref{Table23}(30), $p=5$ is impossible. This completes the proof.
	\end{proof}

	\begin{core}\label{3k2}
		Let $G$ be a finite non-abelian  group of order $2p^k$, $k\geq 2$ and $p>2$ a prime. Then $G$ is Cayley $2$-integral if and only if $G=\langle a,b,c\mid a^2=b^3=c^3=(ab)^2=(ac)^2=[b,c]=1\rangle\cong (\Bbb Z_3\times\Bbb Z_3)\rtimes\Bbb Z_2$.
	\end{core}
	\begin{proof}
One way of the proof is done by Lemma \ref{main3k2}. The converse direction follows from Lemma \ref{remains}(8). 
		\end{proof}

\begin{lem}\label{3k-14}
	Let $G$ be a finite non-abelian group of order $3^k\times 4$, $k\geq 3$. Suppose further that every element of $G$ has order dividing $12$ and every proper normal subgroup of $G$ is abelian. Then $G$
	has no normal subgroup $N$ isomorphic to $\Bbb Z_3^{k-1}\times\Bbb Z_4$.
\end{lem}
\begin{proof}
	Suppose, towards a contradiction, that such a normal subgroup $N$ of $G$ exists. Then $N=H\times K$, where $H\cong\Bbb Z_3^{k-1}$ and $K\cong\Bbb Z_4$. Since $(|H|,|K|)=1$, $H$ and $K$ are both characteristic subgroups of $N$ and so both are normal subgroups of $G$. Indeed $K$ is a Sylow $2$-subgroup of $G$. So $G=PK$, where $P$ is a Sylow $3$-subgroup of $G$. Since $\Aut(K)\cong\Bbb Z_2$, we have $|G:C|=1$ or $2$, where $C=C_G(K)$. On the other hand, $N\leq C$ and $|G:N|=3$. This implies that $G=C$, which means $K\leq Z(G)$. Thus $P\unlhd G$, and so $G=P\times K$ is abelian, a contradiction.
\end{proof}
\begin{lem}\label{3k-122}
	Let $G$ be a finite non-abelian group of order $3^k\times 4$, $k\geq 2$. Suppose further that every element of $G$ has order dividing $12$ and every proper normal subgroup of $G$ is abelian. Then $G$
	has no normal subgroup $N$ isomorphic to $\Bbb Z_3^{k-1}\times\Bbb Z_2^2$.
\end{lem}
\begin{proof}
%	Suppose, towards a contradiction, that such a normal subgroup $N$ of $G$ exists. Then $N=H\times K$, where $H\cong\Bbb Z_3^{k-1}$ and $K\cong\Bbb Z_2^2$. Clearly, $K$ is the only Sylow $2$-subgroup of $G$. Also there exists a Sylow $3$-subgroup $P$ of $G$ such that $H\leq P$. Furthermore, $H\unlhd P$ and $N\cap P=H$. Thus $|NP|=3^{k}\times 4=|G|$, which means $G=NP$. Since $N$ and $P$ normalize $H$, we have $H\unlhd G$. Thus $G/H$ is a group of order $12$ having $N/H\cong \Bbb Z_2^2$ as a normal subgroup.
%	We claim that $G/H$ is non-abelian. To see this, let $G/H$ be abelian. Then $G/H\cong A\times B$, where $A=\langle aH\rangle\cong \Bbb Z_3$ for some $a\in G\setminus H$ and
%	$B\cong \Bbb Z_2^2$. Let $x\in K$, then $ax=xa$, because $G'\leq H$ and $K\unlhd G$ which imply $axa^{-1}x^{-1}\in H\cap K=1$. This means that $K$ centralizes $L:=\langle H,a\rangle$. Since $L=H\langle a\rangle=H\cup Ha\cup Ha^2$, we have $|L|=3^k$ and so $G=LK=L\times K$ is abelian, since $L$ is abelian by our assumption. This is a contradiction. Thus our claim is true and $G/H$ is a non-abelian group of order $12$ having a normal subgroup isomorphic to $\Bbb Z_2^2$. Thus $G/H\cong A_4$.
Suppose, towards a contradiction, that such a normal subgroup $N$ of $G$ exists. Then $N=H\times K$, where $H\cong\Bbb Z_3^{k-1}$ and $K\cong\Bbb Z_2^2$. Since $|H|$ and $|K|$ are coprime, $H$ and $K$ are characteristic subgroups of $N$, which means $H\unlhd G$ and $K\unlhd G$. Then $G/K$ is a group of order $3^k$. Since every element of $G$ has order dividing $12$, we conclude that $G/K$ is an elementary abelian $3$-group, i.e $G/K\cong\Bbb Z_3^k$. Thus  $P$, a Sylow $3$-subgroup of $G$ acts on $K$ and so by Fitting's Lemma  either $G\cong H\times (K\rtimes \Bbb Z_3)\cong H\times A_4$ which is impossible by Table \ref{Table23}(3), or $G$ is abelian, a contradiction.  
\end{proof}

% On the other hand, by the Schur-Zassenhaus theorem, there exist a subgroup $L$ of $G$ such that $G=KL$ and $K\cap L=1$. Thus $L\cong G/K\cong\Bbb Z_3^k$. Since $L$ acts on $K$, by conjugation, there exists a group homomorphism $\varphi:L\rightarrow \Aut(K)$ such that $\ker\varphi=L\cap C_G(K)$. If $\ker\varphi=L$ then $G=K\times L$ is abelian, which is a contradiction. Hence $M:=\ker\varphi$ is a proper subgroup of $L$, since $k\neq 1$. On the other hand, $L/M$ is a $3$-group and isomorphic to a subgroup of $\Aut(K)\cong S_3$. Thus $L/M\cong \Bbb Z_3$. Since $L$ is an elementary abelian $3$-group, there exists a subgroup $M'$ of $L$ such that $L=M'\times M$. Then $G=K(M'\times M)=KM'\times M$, because $K\cap L=1$, $L=M'M$ is abelian and $M$ centralizes $K$. Now $K\cap M'=1$ implies that $KM'$ is a group of order $12$. Since $\Bbb Z_2^2\cong K\unlhd KM'$, we conclude that $KM'\cong A_4$. Thus $G$ has a non-abelian normal subgroup isomorphic to $A_4$, a contradiction.

\begin{lem}\label{conj}
	Let $G$ be a non-abelian group of order $3^k\times 4$, where $k\geq 2$. Suppose further that every proper normal subgroup of $G$ is abelian and each element of $G$ has order dividing $12$. If $G$ has a normal subgroup isomorphic to 
	$\Bbb Z_3^k\times\Bbb Z_2$ then 
	\[G=\langle a_1,\ldots,a_k,b\mid a_i^3=b^4=[a_i,a_j]=ba_ib^{-1}a_i=1, i,j=1,\ldots,k\rangle\cong\Bbb Z_3^k\rtimes\Bbb Z_4.\]
\end{lem}
\begin{proof}
	Let $N=P\times R$, where $P=\langle a_1,\ldots,a_k\rangle\cong\Bbb Z_3^k$ and $R=\langle a\rangle\cong\Bbb Z_2$, be the normal subgroup of $G$ isomorphic to $\Bbb Z_3^k\times\Bbb Z_2$. Since $|P|$ and $|R|$ are coprime, they are characteristic subgroups of $N$ and so both are normal subgroups of $G$. Furthermore, $P$ is the Sylow $3$-subgroup of $G$. Let $Q$ be a Sylow $2$-subgroup of $G$. Then $R\leq Q$, $|Q|=4$,  $G=PQ$ and $P\cap Q=1$. Since $R$ is a normal subgroup of $G$ of order $2$, we have $R\leq Z(G)$. Let $b\in Q\setminus R$. Then  $Q=\langle b,a\rangle\cong\Bbb Z_2^2$ or $\Bbb Z_4$. Furthermore $G=N\langle b\rangle$. If $Q\cong\Bbb Z_2^2$ then $P\langle b\rangle$ is a proper normal subgroup of $G$. Thus $b$ centralizes $P$ which implies $Q$ does. This means $G=P\times Q$ is abelian, a contradiction. Hence $Q\cong\Bbb Z_4$ and more precisely $Q=\langle b\rangle$, $a=b^2$ and $G=P\langle b\rangle$.
	
	Let $[P,b]=\langle [x,b]\mid x\in P\rangle$, where $[x,b]=x^{-1}b^{-1}xb$. Then $[P,b]<P$ or $[P,b]=P$. In the first case, $[P,b]\langle b\rangle$ is a proper normal subgroup of $G$, and is abelian by our assumption. Thus $b$ commutes with each element of $[P,b]$. So for all $x\in P$, we have
	\[[x,b]^2=x^{-1}b^{-1}xbx^{-1}b^{-1}xb=x^{-1}b^{-1}xx^{-1}b^{-1}xb^2=x^{-1}b^{-2}xb^2=x^{-1}axa=1,\] 
	and so $[x,b]=1$, because every non-identity element of $P$ has order $3$. This means $b$ commutes with all elements of $P$ and therefore $G=P\times\langle b\rangle$ is abelian, a contradiction. Thus
	$[P,b]=P$. Let $i\in\{1,\ldots,k\}$. Let $x\in P$, then 
	\[b^{-1}[x,b]b=b^{-1}(x^{-1}b^{-1}xb)b=(b^{-1}x^{-1}b)b^{-2}xb^2=b^{-1}x^{-1}baxa=b^{-1}x^{-1}bx=[x,b]^{-1},\]
	which implies that for all $i\in\{1,\ldots,k\}$, $b^{-1}a_ib=a_i^{-1}$. This completes the proof.
\end{proof}

	\begin{lem}\label{1088}
	There is no non-abelian Cayley $2$-integral group of order $108$ containing a normal subgroup (maximal among normal subgroups) isomorphic to $G_{36,7}\cong (\Bbb Z_3\times\Bbb Z_3)\rtimes\Bbb Z_4$.
\end{lem}
\begin{proof}
	Let $G$ be a non-abelian Cayley $2$-integral group of order $108$, and $M$ be a maximal normal subgroup of $G$ isomorphic to  $G_{36,7}$. Then, by GAP library, $G$ is one of the groups $G_{108,6}$ or $G_{108,29}$. By GAP, $G_{108,6}$ has a normal subgroup $N$ of order $27$ which is not elementary abelian, which contradicts Lemma \ref{elab}. Also $G_{108,29}$ has a normal subgroup $N$ such that $G/N\cong S_3\times\Bbb Z_3$, which contradicts Corollary \ref{3k2}.
\end{proof}

\begin{lem}\label{36-18} There is no non-abelian Cayley $2$-integral group of order $36$ containing a  normal subgroup isomorphic to $G_{18,4}\cong (\Bbb Z_3\times\Bbb Z_3)\rtimes\Bbb Z_2$.
\end{lem}
\begin{proof}
	Suppose, towards a contradiction, that $G$ is a non-abelian Cayley $2$-integral group of order $36$ containing a normal subgroup isomorphic to $G_{18,4}$. By GAP library, $G$ is one of the
	groups $G_{36,9}$,  $G_{36,10}$ or $G_{36,13}$. By the rows $(31)$ and $(32)$ of Table \ref{Table23}, we may assume that $G\cong G_{36,10}$. Then $G$ has a subgroup isomorphic to $G_{18,4}\cong S_3\times\Bbb Z_3$ which is impossible by Corollary \ref{3k2}. 
\end{proof}

	\begin{lem}\label{108}
	There is no non-abelian Cayley $2$-integral group of order $108$.
\end{lem}
\begin{proof}
	By GAP, there are exactly $39$ non-abelian groups  $L$ of order $108$. Among them, the only groups that order of each element is in $\{1,2,3,4,6,8,12\}$
	are $G_{108,r}$, where 
	$$r\in\{8,11,13,15,17,22,25,28,30,32,33,34,36,37,38,39,40,41,42,43,44\}.$$
	If $r=8$ or $33$, then $L$ has a normal subgroup isomorphic to $G_{36,7}$, which contradicts Lemma \ref{1088}. If $r=11,13$, or $30$, then $L$ has a subgroup isomorphic to $G_{54,10}$, and if  $r$ is one of the integers $15,17,25,28,36,37,38,39,40,42$ or $43$, then $L$ has a subgroup isomorphic to $G_{18,3}$, which contradicts Corollary \ref{3k2}. If $r=22$ or $41$, then $L$ has a subgroup isomorphic to $G_{36,11}$, if $r=32$ then $L$ has a subgroup isomorphic to $G_{36,6}$, and if $r=44$ then $G$ has a subgroup isomorphic to $G_{36,13}$. The groups $G_{36,6}$, $G_{36,11}$ and $G_{36,13}$ are not $2$-integral, by Table \ref{Table23}, a contradiction. If $r=34$ then $L$ has
	a normal subgroup $N$ such that $L/N$ is isomorphic to $G_{54,14}$ which will be Cayley $2$-integral by Lemma \ref{sbg-qoutient}, a contradiction by Corollary \ref{3k2}. This completes the proof. 
\end{proof} 

Recall that by GAP notation, $G_{36,7}=\langle a_1,a_2,b\mid a_1^3=a_2^3=b^4=[a_1,a_2]=1, a_1^b=a_1^{-1}, a_2^b=a_2^{-1}\rangle$. In the following result, we prove that the only non-abelian Cayley $2$-integral group of order $3^k\times 4$, $k\geq 2$ is $G_{36,7}$.
\begin{core}\label{36main}
	Let $G$ be a finite non-abelian group of order $3^k\times 4$, $k\geq 2$. Then $G$ is Cayley $2$-integral if and only if $G\cong G_{36,7}$.
\end{core}
\begin{proof}
	First suppose, towards a contradiction, that $G\ncong G_{36,7}$ is a  Cayley $2$-integral group which has minimal order among all non-abelian Cayley $2$-integral groups of order $3^k\times 4$, where $k\geq 2$. Since $G$
	is a solvable group, it is not a simple group and so it has at least one proper normal subgroup.
	We deal with the following cases:
	
	\textbf{Case 1.} $G$ has a proper maximal normal subgroup (maximal among normal subgroups) $M$,  which is non-abelian. Then $G/M\cong \Bbb Z_2$ or $\Bbb Z_3$, since $G/M$ is simple and solvable. In the first case, $M$ is a non-abelian Cayley 
	$2$-integral group of order $3^k\times 2$. Then by Corollary \ref{3k2}, $k=2$, $M\cong G_{18,4}$ and $|G|=36$, which contradicts Lemma \ref{36-18}. So the latter case holds, which means $|M|=3^{k-1}\times 4$. Since $M$ is also Cayley $2$-integral group, our assumption on $G$ and Lemma \ref{108} imply that $M$ is abelian. 
		
	\textbf{Case 2.}  Every proper maximal normal subgroup of $G$ is abelian. Then every proper normal subgroup of $G$ is abelian. Let $M$ be a maximal normal subgroup of $G$. Then $M$ is abelian and again  $G/M\cong\Bbb Z_2$ or $\Bbb Z_3$. First suppose that  $G/M\cong\Bbb Z_3$ then $|M|=3^{k-1}\times 4$. Since $M$ is Cayley $2$-integral and abelian,  Proposition \ref{baseab} implies that $M\cong\Bbb Z_3^{k-1}\times\Bbb Z_2^2$ or $\Bbb Z_3^{k-1}\times\Bbb Z_4$, which is a contradiction by Lemma \ref{3k-122} and Lemma \ref{3k-14}, respectively, because each element of $G$ has order dividing $12$ by Corollary \ref{2,3,5}. Hence we may assume that $G/M\cong\Bbb Z_2$. Thus $M\cong\Bbb Z_3^k\times\Bbb Z_2$ by Proposition \ref{baseab}. Now Lemma \ref{conj} implies that $G$ has the following presentation:
	\[G=\langle a_1,\ldots,a_k,b\mid a_i^3=b^4=[a_i,a_j]=ba_ib^{-1}a_i=1, i,j=1,\ldots,k\rangle\cong\Bbb Z_3^k\rtimes\Bbb Z_4.\]
	If $k\geq 3$ then $\langle a_1,a_2,a_3,b\rangle$ is a subgroup of $G$ of order $108$ which is Cayley $2$-integral. This contradicts Lemma \ref{108}. Thus $k=2$ and $G\cong G_{36,7}$, a contradiction.
	This proves one direction. The converse direction follows from Table \ref{Table23}(12).
	\end{proof}

\begin{lem}\label{akhlaghi1}
	Let $G$ be a finite non-abelian  $\{2,3\}$-group which is not a $p$-group. If $G$ is a Cayley $2$-integral group, then either $G\cong A_4$ or $G$ has a normal Sylow $3$-subgroup.
\end{lem}
\begin{proof}
Let $|G|=2^a 3^b$, where $a,b\geq 1$ and $P$ be a Sylow $3$-subgroup of $G$. Then Lemmas \ref{elab} and \ref{sbg-qoutient} imply that $P$ is an elementary abelian $3$-group.	Assume $G$ is a counterexample with minimal order, and $N$ is a minimal normal subgroup of $G$. Then either $G/N\cong A_4$ or $G/N$ has a normal Sylow $3$-subgroup. 
	
	First we assume that $G/N\cong A_4$. Since $G$ is a solvable group, $N$ is a $p$-group, where $p=2$ or $3$. If $p=3$ then $P\leq C_G(N)$. Since $PN/N\subseteq C_G(N)/N\unlhd G/N\cong A_4$, we have $G=C_G(N)$. Thus we may assume that $|N|=3$ and by the classification of groups of order $48$ we get that $G=N\times H$, where $H\cong A_4$,  which contradicts Table \ref{Table23}(3). Thus $N$ is a $2$-group which implies $|P|=3$. 
	Note that $N\cap Z(Q)\not =1	$ is a normal subgroup of $G$, where $Q$ is the normal Sylow $2$-subgroup of $G$. Hence $N\subseteq Z(Q)$.  If there exists any element $x\in N$ such that $3$ divides $|C_G(x)|$, then $x\in Z(G)$ and so $N=\langle x\rangle\cong\Bbb Z_2$. Hence according to the classification of the groups of order 24, we have  $G\cong\Bbb Z_2\times A_4$, which is impossible by Table \ref{Table23}(2). Thus for all $x\in N$, we have $C_G(x)\leq Q$.  Since the conjugation action of a Sylow $3$-subgroup of $A_4$ on its Sylow $2$-subgroup is Frobenius, then  for all $x\in Q-N$  we deduce that  	
	$C_G(x)\subseteq Q$ and so  $G=Q\rtimes P$ is a Frobenius group, by \cite[Theorem 6.4]{Issacs}. Thus $3$ divides $2^a-1$, which means $a$ is even. If $Q$ is non-abelian, then by Proposition \ref{non-ab-2groups}, either $|Q|=16$ and $Q\ncong D_{16}, QD_{16}$ or $Q\cong Q_8\times\Bbb Z_2^{a-3}$. Thus $Q'$ or $Z(Q)$ is a cyclic group of order $2$. If $Z(Q)$ is not a cyclic group, set $M=Q'$ and otherwise set $M=Z(Q)$. Then $M\unlhd G$. So $MP$ (a group of order $6$ with a normal subgroup of order $2$) can not be a Frobenius group, a contradiction. Thus, we may assume that \( Q \) is abelian. Then, by Proposition \ref{baseab}, \( Q \) is isomorphic to \( \mathbb{Z}_2^{a_1} \times \mathbb{Z}_4^{a_2} \times \mathbb{Z}_8^{a_3} \), where \( a_1, a_2, a_3 \geq 0 \). Given that \( G/N \cong A_4 \), it is clear that \( Q/N \cong \mathbb{Z}_2^2 \). This implies that \( Q \cong \mathbb{Z}_2^{a_1} \times \mathbb{Z}_4^{a_2} \) and \( \Phi(Q) \cong \mathbb{Z}_2^{a_2} \). 
	
	First, let \( a_2 = 0 \). If \( a_1 > 2 \), then \( G \) has a Frobenius factor isomorphic to \( \mathbb{Z}_2^4 \rtimes \mathbb{Z}_3 \), which is Cayley \( 2 \)-integral, leading to a contradiction as shown in Table \ref{Table23}(5). Therefore, we conclude that \( a_1 = 2 \), and thus \( G \cong A_4 \), which again is a contradiction. Hence, we have \( a_2 \neq 0 \). 
	
	Since \( \Phi(Q) \leq N \), it follows that \( N = \Phi(Q) \cong \mathbb{Z}_2^{a_2} \). Consequently, \( Q/N \cong \mathbb{Z}_2^{a_1 + a_2} \), leading to \( a_1 + a_2 = 2 \). This implies that \( a_2 \) can be \( 1 \) or \( 2 \). In the first case, \( NP = N \rtimes P \) becomes a Frobenius group, which is impossible. Hence, we may assume \( a_2 = 2 \), meaning \( N \cong \mathbb{Z}_2^2 \) and \( Q \cong \mathbb{Z}_4^2 \). Then, \( G \cong \mathbb{Z}_4^2 \rtimes \mathbb{Z}_3 \) becomes a Cayley \( 2 \)-integral group, contradicting Table \ref{Table23}(4). 
	
	Now, we assume that \( G/N \) has a normal Sylow \( 3 \)-subgroup. Then \( PN/N \unlhd G/N \), which implies \( PN \unlhd G \). If \( N \) is a \( 3 \)-group, then \( PN = P \) and thus \( P \unlhd G \), leading to a contradiction. Therefore, we have demonstrated that every minimal normal subgroup \( N \) of \( G \) is an elementary abelian \( 2 \)-group and \( PN \unlhd G \).
	
	Let \( 1 \neq x \in P \). If \( x \in C_G(N) \), then \( |C_G(N)| \) is divisible by \( 3 \). Let \( K \) be a Sylow \( 3 \)-subgroup of \( C_G(N) \). Then \( KN/N \) is a characteristic subgroup of \( C_G(N)/N \), and since \( C_G(N)/N \) is normal in \( G/N \), we conclude that \( KN/N \) is normal in \( G/N \), implying that \( NK = N \times K \) is a normal subgroup of \( G \). Thus, \( G \) has a minimal normal subgroup that is a \( 3 \)-group, which is impossible.
	
	Now set \( H = \langle N, x \rangle = N \rtimes \langle x \rangle \). By  Fitting's Lemma, we have \( N = [N, \langle x \rangle] \times C_N(x) \). This implies that \( H = N_2 \times (N_1 \rtimes \langle x \rangle) \), where \( N_1 = [N, \langle x \rangle] \) and \( N_2 = C_N(x) \). As \( x \notin C_G(N) \), it follows that \( N_2 \neq N \). If \( |N_1| > 4 \), then \( H \) contains a Frobenius subgroup isomorphic to \( \mathbb{Z}_2^4 \rtimes \mathbb{Z}_3 \), leading to a contradiction. Therefore, \( |N_1| = 4 \). If \( N_2 > 1 \), then \( H \) would have a subgroup isomorphic to \( \mathbb{Z}_2 \times( \mathbb{Z}_2^2 \rtimes \mathbb{Z}_3 )\), again resulting in a contradiction. Thus, we have \( N_1 \cong \mathbb{Z}_2^2 \) and \( N_2 = 1 \). Therefore, \( N = N_1 \) and \( H = N_1 \rtimes \langle x \rangle \cong A_4 \), forming a Frobenius group. 
	
	The previous discussion applies for all \( 1 \neq x \in P \), implying that \( N \rtimes P \) is a Frobenius group. Since \( P \) is an elementary abelian \( 3 \)-group, we conclude that \( P = \langle x \rangle \) by \cite[Corollary 6.17]{Issacs}. Therefore, \( H \unlhd G \) and \( G/H \) is a \( 2 \)-group. Now, \( G \) contains a normal subgroup \( K \), which is Cayley \( 2 \)-integral, such that \( H < K \) and \( K/H \cong \mathbb{Z}_2 \). Thus, \( K \cong A_4 \times \mathbb{Z}_2 \) or \( S_4 \), contradicting Table \ref{Table23}(2) or Table \ref{Table23}(1), respectively. This completes the proof.\end{proof}
	
%	Let $C=C_G(N)$. If $3$ divides $|C|$, then $P_0:=P\cap C\neq 1$ is a Sylow $3$-subgroup of $C$.  

%	Put $H=NP$. If $H$ is a Frobenius group, then $N\cong\Bbb Z_2^2$ and $P\cong\Bbb Z_3$, which means $H\cong A_4$. Then $G/H$ is a $2$-group and so $G$ has a subgroup $K$ such that 
%	$H<K$ and $K/H\cong\Bbb Z_2$. Thus $K\cong A_4\times\Bbb Z_2$ or $S_4$ is Cayley $2$-integral, a contradiction by Lemma \ref{selected}.

%	???????We claim that the action of $P$ on $N$ is Frobenious.  Let this action is not Frobenius. Then there exists $1\neq x\in P$ such that $C_H(x)\nsubseteq P$, where $H=NP$. Thus there exists $1\neq a\in N$ and $y\in P$ such that $ayx=xay$. Since $xy=yx$, we have $a\in C_N(x)$. On the other hand, by Fitting Theorem \cite[Theorem 4.34]{Issacs}, $N=N_1\times N_2$, where $N_1=[N,\langle x\rangle]$ and $N_2=C_N(x)$. This implies that 
	
%	 Put $H=N\langle x\rangle$. Then $H=N_2\times H_1$, where $H_1=N_1\rtimes \langle x\rangle$. Let $x\notin C_G(N)$. Then $H_1$ is a Frobenious group and so $N_1\cong\Bbb Z_2^{2a}$ for some $a\geq 1$. If $a\geq 2$ then $H_1$ has a quotient isomorphic to $\Bbb Z_2^4\rtimes\Bbb Z_3$ which is Cayley $2$-integral, a contradiction. Hence $N_1\cong\Bbb Z_2^2$. If $|N_2|>1$ then $H$ has a subgroup isomorphic to $\Bbb Z_2\times A_4$, which is not a Cayley $2$-integral group, a contradiction. Thus $N\cong\Bbb Z_2^2$ and $H\cong A_4$.
	
%	If $C_G(N)\cap P=1$ then 

\begin{lem}\label{upto48}
	Let $G$ be a finite non-abelian $\{2,3\}$-group of order $|G|\leq 48$, which is not a $p$-group. Then $G$ is Cayley $2$-integral if and only if $G$ is one of the groups
	\begin{itemize}
		\item[(1)] $S_3$,
		\item[(2)] $G_{12,1}\cong Dic_{12}$,
		\item[(3)] $G_{12,3}=\langle a,b\mid   b^2=a^3=(ba)^3=1 \rangle\cong A_4$,
		\item[(4)] $G_{12,4}\cong D_{12}$,
		\item[(5)] $G_{18,4}=\langle a,b,c\mid  a^2=b^3=c^3=(ab)^2=(ac)^2=c^{-1}c^b=1 \rangle\cong (\Bbb Z_3\times \Bbb Z_3)\rtimes \Bbb Z_2$,
		\item[(6)] $G_{24,4}=\langle a,b,c\mid  c^3=b^2a^2=baba^{-1}=a^4=c^ac=c^{-1}c^b=(a^{-1}c)^2a^{-2}=1 \rangle \cong\Bbb Z_3\rtimes Q_8$,
		\item[(7)] $G_{24,11}=\langle a,b,c\mid  c^3=b^2a^2=baba^{-1}=a^4=c^{-1}c^a=c^{-1}c^b=1\rangle\cong Q_8 \times\Bbb Z_3 $,
		\item[(8)] $G_{24,7}=\langle a,b,c\mid b^2=c^3=a^4=c^ac=bb^a=b^cb=1 \rangle\cong\Bbb Z_2\times(\Bbb Z_3\rtimes \Bbb Z_4)$,
		\item[(9)] $G_{36,7}=\langle a,b,c\mid  b^3=c^3=a^4=b^ab=c^ac= c^{-1}c^b=1 \rangle \cong(\Bbb Z_3\times \Bbb Z_3)\rtimes \Bbb Z_4$.
	\end{itemize}
\end{lem}
\begin{proof}
One direction follows from Corollary \ref{coremain}. Conversely, suppose that $G$ is a Cayley $2$-integral group not given in $(1)$-$(9)$. Then, by Corollary \ref{36main}, $|G|\neq 36$ and so $|G|\in\{18,24,48\}$.  By Table \ref{Table23}, $G$ is not a dihedral group. Also, by Table \ref{Table23}, $G\ncong H$, where  $H:= S_{4}, A_4\times\Bbb Z_2, A_4\times\Bbb Z_3, \Bbb Z_2^4\rtimes\Bbb Z_3, \Bbb Z_4^2\rtimes\Bbb Z_3$, or $ \Bbb Z_3\rtimes\Bbb Z_8$. Furthermore, $G\ncong K:=\textrm{SL}(2,3)$ by Lemma \ref{akhlaghi1}. This implies that if $|G|\neq 48$ then $G$ is one of the groups listed in the rows $(9)$-$(13)$ of Table \ref{Table23}, a contradiction. Hence we may assume that $|G|=48$. Then $G$ has no subgroup isomorphic to $H$ or $K$.Now, by Lemma \ref{akhlaghi1} and Corollary \ref{2,3,5}, the only possibilities for $G$ are the groups given in the rows $(14)$-$(22)$ of Table \ref{Table23}, a contradiction and completing the proof. 
\end{proof}
\begin{pro}\label{2,3}
	Let $G$ be a finite non-abelian $\{2,3\}$-group which is not a $p$-group. Then $G$ is Cayley $2$-integral if and only if $G$ is one of the groups given in Lemma \ref{upto48}.
%	\begin{itemize}
%		\item[(1)] $S_3$
%		\item[(2)] $G_{12,1}\cong Dic_{12}$
%		\item[(3)] $G_{12,3}=\langle a,b\mid   b^2=a^3=(ba)^3=1 \rangle\cong A_4$
%		\item[(4)] $G_{12,4}\cong D_{12}$
%		\item[(5)] $G_{18,4}=\langle a,b,c\mid  a^2=b^3=c^3=(ab)^2=(ac)^2=c^{-1}c^b=1 \rangle\cong (\Bbb Z_3\times \Bbb Z_3)\rtimes \Bbb Z_2$
%		\item[(6)] $G_{24,4}=\langle a,b,c\mid  c^3=b^2a^2=baba^{-1}=a^4=c^ac=c^{-1}c^b=(a^{-1}c)^2a^{-2}=1 \rangle \cong\Bbb Z_3\rtimes Q_8$
%		\item[(7)] $G_{24,11}=\langle a,b,c\mid  c^3=b^2a^2=baba^{-1}=a^4=c^{-1}c^a=c^{-1}c^b=1\rangle\cong Q_8 \times\Bbb Z_3 $
%		\item[(8)] $G_{24,7}=\langle a,b,c\mid b^2=c^3=a^4=c^ac=bb^a=b^cb=1 \rangle\cong\Bbb Z_2\times(\Bbb Z_3\rtimes \Bbb Z_4)$
%		\item[(9)] $G_{36,7}=\langle a,b,c\mid  b^3=c^3=a^4=b^ab=c^ac= c^{-1}c^b=1 \rangle \cong(\Bbb Z_3\times \Bbb Z_3)\rtimes \Bbb Z_4$
%	\end{itemize}
\end{pro}
\begin{proof}
	Let $\Omega$ be the set of all nine groups $(1)$-$(9)$ given in Lemma \ref{upto48}. If $G \in \Omega$, then $G$ is a Cayley $2$-integral group by Lemma \ref{upto48}. Conversely, suppose that $G \notin \Omega$ is a Cayley $2$-integral group of minimal order. Let $P$ be a Sylow $3$-subgroup and $Q$ be a Sylow $2$-subgroup of $G$. By Lemma \ref{akhlaghi1}, we have that $G = P \rtimes Q$. Let $x \in Z(Q)$ be an involution. Then $M = P \rtimes \langle x \rangle$ is a normal subgroup of $G$, and so we can express $M$ as $M = P_1 \times (P_2 \rtimes \langle x \rangle)$, where $P_1 = C_P(x)$ and $P_2 = [P,x]$, according to Fitting's Lemma. Now, by Corollary \ref{3k2}, we conclude that either $M = P \times \langle x \rangle$, $M \cong S_3$, or $M \cong \mathbb{Z}_3^2 \rtimes \mathbb{Z}_2$. 
	
	First, suppose that $M = P \times \langle x \rangle$. Then $\langle x \rangle \unlhd G$, and the quotient $G/\langle x \rangle$ is either abelian or a Cayley $2$-integral group isomorphic to one of the nine groups specified. Assume first that $G \in \Omega$. Thus, we have $|G/\langle x \rangle| \leq 24$ or $G/\langle x \rangle \cong G_{36,7} \cong \mathbb{Z}_3^2 \rtimes \mathbb{Z}_4$. In the former case, $|G| \leq 48$. Therefore, by Lemmas \ref{upto48}, $G$ is one of the groups $(1)$-$(9)$, which leads to a contradiction. In the latter case, $G \cong \mathbb{Z}_3^2 \rtimes \mathbb{Z}_8$ or $(\mathbb{Z}_3^2 \rtimes \mathbb{Z}_4) \times \mathbb{Z}_2$. The former case is impossible, as $G$ has a Cayley $2$-integral factor group isomorphic to $\mathbb{Z}_3 \rtimes \mathbb{Z}_8$, which contradicts Table \ref{Table23}(6). In the latter case, $G$ contains a factor isomorphic to $(\mathbb{Z}_3^2 \rtimes \mathbb{Z}_2) \times \mathbb{Z}_2$, which must also be Cayley $2$-integral, contradicting Corollary \ref{36main}. Therefore, we may assume that $G/\langle x \rangle$ is abelian. Noting that $[P,Q]\subseteq \langle x\rangle \cap P=1$, we conclude  that $G \cong P \times Q$, and by Lemma \ref{nilpotent}, we obtain that $G \cong \mathbb{Z}_3 \times Q_8 \in \Omega$, leading to a contradiction.
	
	Now,  assume that $M \cong S_3$. Then $P \cong \mathbb{Z}_3$, and the quotient $G/C_G(P)$ is isomorphic to a subgroup of $\text{Aut}(P) \cong \mathbb{Z}_2$. Therefore, we conclude that $G/C_G(P) \cong \mathbb{Z}_2$. If $2$ divides $|C_G(P)|$, then $C_G(P) = P \times K$, where $K \neq 1$ is a Sylow $2$-subgroup of $C_G(P)$. This implies that $K \unlhd G$. Since $x \notin C_G(P)$, we have $x \notin K$, and as $x \in Z(Q)$, we get that $Q = K \times \langle x \rangle$. This implies that $G \cong S_3 \times K$. If $|K| = 2$, then $|G| = 12$, and hence $G \in \Omega$, which is a contradiction. If $|K| \geq 4$, then $Q$ has a subgroup of order $8$, leading $G$ to have a Cayley $2$-integral subgroup $L \times S_3$, where $L \leq K$ has order $4$, which is again a contradiction. Thus, we conclude that $C_G(P) = P$, which implies $G \cong S_3$, a contradiction.
	
	Finally, suppose that $M \cong \mathbb{Z}_3^2 \rtimes \mathbb{Z}_2$. Then $P \cong \mathbb{Z}_3^2$, and $G/M$ is a non-trivial $2$-group. Let $gM \in G/M$ be an involution. Then $g^2 \in M$. Let $L = \langle M, g \rangle$. Then $L = M \langle g \rangle \leq G$ is a Cayley $2$-integral group, has order $36$, and contains $M$. This contradicts Lemma \ref{36-18}.
\end{proof}

\subsection{Non-abelian Cayley $2$-integral $\{2,5\}$-groups.}
\begin{pro}\label{2,5}
	Let $G$ be a finite non-abelian $\{2,5\}$-group which is not a $p$-group. Then $G$ is Cayley $2$-integral if and only if $G$ is one of the following groups
	\begin{itemize}
		\item[(1)] $D_{10}$,
		\item[(2)] $Dic_{20}=\langle x,y\mid x^5=y^4=1, x^y=x^{-1}\rangle\cong\Bbb Z_5\rtimes\Bbb Z_4$.
	\end{itemize}
\end{pro}
\begin{proof}
$D_{10}$ and $Dic_{20}$ are Cayley 2-integral by Corollary \ref{coremain}. Let $|G|=2^n5^m$, where $n,m\geq 1$, and let $G$ be a Cayley 2-integral group of minimal order such that $G \ncong D_{10}$ and $G \ncong Dic_{20}$. Let $P$ be a Sylow 2-subgroup and $Q$ be a Sylow 5-subgroup of $G$. Let $N$ be a minimal normal subgroup of $G$. Then, either $G/N \cong D_{10}$, $G/N \cong Dic_{20}$, $G/N$ is a 2-group, $G/N$ is a 5-group, or $G/N$ is abelian.

First, assume $G/N \cong D_{10}$ or $G/N \cong Dic_{20}$.

Suppose $N$ is a 5-group. Then $N \leq Q$, which means we have $G = Q \rtimes P$. Note that $Q$ is an elementary abelian 5-group, according to Lemma \ref{elab}. By Fitting's Lemma, we have $Q = [Q, P] \times C_Q(P)$. Based on the structure of $G/N$, we find that $C_Q(P) \leq N$. On the other hand, since $C_Q(P) \unlhd G$, it follows that either $C_Q(P) = 1$ or $C_Q(P) = N$. In the first case, $G = [Q, P] \rtimes P$, leading to a quotient isomorphic to $\mathbb{Z}_5^2 \rtimes \mathbb{Z}_2$ or $\mathbb{Z}_5^2 \rtimes \mathbb{Z}_4$, which is impossible by Lemma \ref{main3k2}. In the latter case, we have $G = N \times ([Q, P] \rtimes P) \cong N \times G/N$, resulting in a quotient isomorphic to $\mathbb{Z}_5 \times G/N$, which is also impossible by Lemma \ref{main3k2}.

Now let $N$ be a 2-group. Then $G$ has a proper normal subgroup $M$ such that $N < M$ and $M = QN \cong N \rtimes Q$. Thus, $M = C_N(Q) \times ([N,Q] \rtimes Q)$. As $C_N(Q) = Z(M)$ and $N$ is a minimal normal subgroup of $G$, we find that either $C_N(Q) = 1$ or $C_N(Q) = N$. In the first case, by induction, we reach a contradiction. Thus, we conclude that $M = N \times Q$. Now, by considering the group $G/Q \cong P$, we conclude that $|N| = 2$. Therefore, $P \cong \mathbb{Z}_2^2$ or $C_4$ when $G/N \cong D_{10}$, and $P \cong \mathbb{Z}_8$ or $\mathbb{Z}_2 \times \mathbb{Z}_4$ when $G/N \cong Dic_{20}$. Consequently, we have $G \in \{ \mathbb{Z}_2 \times D_{10} \cong D_{20}, Dic_{20}, \mathbb{Z}_2 \times Dic_{20}, \mathbb{Z}_5 \rtimes \mathbb{Z}_8 \}$. By Table \ref{Table23}, none of those groups are Cayley 2-integral, except for $Dic_{20}$, resulting in a final contradiction.
%%%%%%%%%%%%%%%%%%%%%%%%%%%%%%%%%%%%%

Now assume that \(G/N\) is a \(2\)-group. Then we can conclude that \(Q = N\) must be a \(5\)-group. Let \(N < M\) be a normal subgroup of \(G\) with order \(2|N|\). It follows that \(M = N \rtimes L\), where \(L \cong \mathbb{Z}_2\). By applying Fitting's Lemma, we can express \(M\) as \(M = ([N, L] \rtimes L) \times C_N(L) \cong (\mathbb{Z}_5^b \rtimes \mathbb{Z}_2) \times \mathbb{Z}_5^a\) for some integers \(a\) and \(b\). Given the minimality of \(N\), either \(a = 0\) or \(b = 0\).

If \(a = 0\), by Table \ref{Table23}(7), we find that \(b = 1\). This leads us to realize that \(G/C_G(N)\) is a subgroup of \(\mathbb{Z}_4\). If \(N < C_G(N)\), then the Sylow \(2\)-subgroup of \(C_G(N)\), which we denote as \(K\), is normal. Assume \(K_0 \leq K\) is a minimal normal subgroup of \(G\). Consequently, the group \(G/K_0\) will have an order divisible by \(10\), considering \(G\) is not abelian. Note that \(G/K_0\) is not abelian; otherwise, \([N, P] = [Q, P] \subseteq K_0 \cap Q = K_0 \cap Q = 1\) and thus \(G\) would be nilpotent. Since $G$ does not have any element of order $20$ we deduce that \(G\) is abelian. Therefore, by induction, we can assert that \(G/K_0 \cong D_{10}\) or \(Dic_{20}\), a situation we have examined previously. Hence, we may take the assumption that \(N = C_G(N)\), which implies \(G \cong D_{10}\) or \(G \cong Dic_{20}\); this leads to a contradiction. Therefore, we confirm that \(b = 0\).

As a result, \(G/L\) is a group whose order is also divisible by \(10\). If \(G/L\) is abelian, then $[P,Q]\subseteq Q\cap L=1$, and so \(G\) is nilpotent, meaning \(G = P \times Q\). Since \(G\) does not have any element of order \(20\), it also does not have any element of order \(4\), and hence \(G\) must be elementary abelian, which is a contradiction. Given the minimality of \(G\), we conclude that \(G/L \cong D_{10}\) or \(Dic_{20}\), both of which we have previously discussed.

Next, we explore the case where \(G/N\) is a \(5\)-group, implying that \(N=P\) is a \(2\)-group. In this scenario, \(G\) has a subgroup \(N < M\) with order \(5|N|\). Thus, we can express \(M\) as \(M = N \rtimes L\), where \(L \cong \mathbb{Z}_5\). Applying Fitting's Lemma again, we find \(M = ([N, L] \rtimes L) \times C_N(L) \cong (\mathbb{Z}_2^b \rtimes \mathbb{Z}_5) \times \mathbb{Z}_2^a\) for integers \(a\) and \(b\). Given the minimality of \(N\), we conclude that either \(a = 0\) or \(b = 0\). If \(b = 0\), then the order of \(G/L\) must be divisible by \(10\), as otherwise \(G=M\) is abelian. If \(G/L\) is abelian, then \([P,Q]=[N, Q]\subseteq L\cap N=L\cap P=1\), and so \(G\) is nilpotent and hence abelian, leading to a contradiction. Thus, by induction, we find that \(G/L \cong D_{10}\) or \(Dic_{20}\), as previously discussed.

We may assume \(b = 0\), which leads us to conclude that \(M\) and hence \(G\) is a Frobenius group, with \(N\) serving as the Frobenius kernel. Consequently, \(G/N\) must be cyclic, and we can express \(G\) as \(M \cong N \rtimes L \cong \mathbb{Z}_2^b \rtimes \mathbb{Z}_5\). Since \(N\) is a minimal normal subgroup, we find that \(b = 4\). However, this contradicts Table \ref{Table23}(8), which states that the group is not Cayley \(2\)-integral.

Finally, assume that \( G/N \) is abelian and its order is divisible by 10. Let \( M/N \) be the Sylow \( p \)-subgroup of \( G/N \), where \( p \) does not divide \( |N| \). Then, either \( M \) is abelian, or by induction, it is isomorphic to \( D_{10} \) or \( Dic_{20} \). Assume \( M \) is abelian. Then \( M = P \times N \) or \( M = Q \times N \), which implies that one of the Sylow subgroups not containing \( N \) is normal. On the other hand, we know that the Sylow subgroup containing \( N \) is also normal. It follows that \( G \) is nilpotent. Since \( G \) does not have any elements of order 20, we conclude that \( G \) must be abelian, which leads to a contradiction. Therefore, we may assume \( M \cong D_{10} \) or \( Dic_{20} \). Thus, \( |N| = 5 \) and \( P \) is a cyclic group of order 2 or 4. Since \( G/N \) is abelian, if \( |P| = 4 \), then \( G/N \) would have an element of order 20, which is not possible. So, we have \( |P| = 2 \), and hence \( G = Q \rtimes P \cong \mathbb{Z}_5^m \rtimes \mathbb{Z}_2 \). Thus, \( G = \mathbb{Z}_5^a \times (\mathbb{Z}_5^b \times \mathbb{Z}_2) \) for some integers \( a \) and \( b \). Following the same reasoning as discussed in the previous paragraph, we find that \( G \cong D_{10} \), leading to a contradiction.
    \end{proof}

\textbf{Proof of Theorem \ref{main}} It is enough to combine Corollary \ref{key} and Propositions \ref{baseab}, \ref{2,3} and \ref{2,5}.

%In the first case, $G\cong \Bbb Z_2\times D_{10}\cong D_{20}$ which is impossible by Lemma   \ref{dihedral}. In the later case,  $G\cong Dic_{20}$  which is impossible. 

%Now suppose that $G/N\cong Dic_{20}$. First let $N$ is a $5$-group. Then $Q$ is elementary abelian $5$-group and $P=\langle x\rangle\cong\Bbb Z_4$ acts on $Q$. Thus $G=C_Q(P)\times ([Q,P]\rtimes P)$. As $C_Q(P)<N$, we get $C_Q(P)=1$ or $N$. If $C_Q(P)=N$ then $G$ has a subgroup isomorphic to $H\cong \Bbb Z_5\times (\Bbb Z_5\rtimes\Bbb Z_4)$ which is not possible, because $H$ has an element of order $20$, a contradiction by Corollary \ref{order}. Thus $C_Q(P)=1$. Take $H=N\rtimes\langle x^2\rangle$. Then $H=C_N(x^2)\times ([N,\langle x^2\rangle]\rtimes \langle x^2\rangle)$. So by minimality of $G$, either $[N,\langle x^2\rangle]=1$ or $H\cong D_{10}$. In the first case, $x^2\in Z(G)$ and $G/\langle x^2\rangle\cong Q\rtimes\langle x^2\rangle$ is a generalized dihedral group where $|Q|\geq 25$, a contradiction. This implies that $\Bbb Z_5^2\rtimes\Bbb Z_2$ is a Cayley $2$-integral group, a contradiction by Table \ref{Table23}(8).  So $H\cong D_{10}$ and therefore $|N|=5$. Thus $G$ has a subgroup isomorphic to $D_{10}\times\Bbb Z_5$, a contradiction Table \ref{Table23}(7). Thus $N$ is a $2$-group. ???


\begin{thebibliography}{99}
	
	\bibitem{AAFW}
	A. Abdollahi, M. Arezoomand, T. Feng and S. Wang,
	On $2$-integral Cayley graphs,
	Ars Math. Contemp. \textbf{26} (2026), \#P3.06.
	
	\bibitem{AV}
	A. Abdollahi and E. Vatandoost,
	Which Cayley graphs are integral?,
	Electron. J. Combin. \textbf{16} (2009), 1--17.
	
	\bibitem{AJ}
	A. Abdollahi and M. Jazaeri,
	Groups all of whose undirected Cayley graphs are integral,
	European J. Combin. \textbf{38} (2014), 102--109.
	
	\bibitem{Azhvan}
	A. Ahmadi, J. P. Bell and B. Mohar,
	Integral Cayley graphs and groups,
	SIAM J. Discrete Math. \textbf{28(2)} (2014), 685--701.
	
	\bibitem{AT}
	M. Arezoomand and B. Taeri,
	On the characteristic polynomial of $n$-Cayley digraphs,
	Electron. J. Combin. \textbf{20(3)} (2013), \#P.57.
	
	\bibitem{Berkovich1}
	Y. Berkovich,
	\emph{Groups of Prime Power Order}, Vol. 1,
	Walter de Gruyter, Berlin, 2008.
	
%	\bibitem{Berkovich2}
%	Y. Berkovich and Z. Janko,
%	\emph{Groups of Prime Power Order}, Vol. 2,
%	Walter de Gruyter, Berlin, 2008.
	
	\bibitem{Cvet}
	D. Cvetković, M. Doob and H. Sachs,
	\emph{Spectra of Graphs: Theory and Applications},
	3rd edition, Johann Ambrosius Barth Verlag, Heidelberg--Leipzig, 1995.
	
	\bibitem{DS}
	P. Diaconis and M. Shahshahani, Generating a random permutation with random transpositions, Z. Wahsch. Verw. Gebiete, \textbf{57} (1981), 159--179.
	\bibitem{Gap}
	The GAP Group, GAP-Groups, Algorithms, and Programming, Version 4.16.0; 2026. (https://www.gap-system.org)

	
		\bibitem{HS}
	F. Harary and A. J. Schwenk,
	Which graphs have integral spectra?,
	in: \emph{Graphs and Combinatorics}
	(Proc. Capital Conf., George Washington Univ., Washington, D.C., 1973),
	Lecture Notes in Math., Vol. 406,
	Springer, Berlin (1974), 45--51.
	
	\bibitem{Herzog}
	M. Herzog,
	On finite simple groups of order divisible by three primes only,
	J. Algebra \textbf{10} (1968), 383--388.
	

	
%	\bibitem{Hof}
%	A. J. Hoffman,	On the polynomial of a graph,
%	Amer. Math. Monthly \textbf{70} (1963), 30--36.
	
	\bibitem{HLM}
	X. Huang, L. Lu and K. Monius, Splitting fields of mixed Cayley graphs over abelian groups, ¨
	J. Algebr. Comb. \textbf{58} (2023), 681--693,
	\bibitem{Issacs}
	I. M. Isaacs,
	\emph{Finite Group Theory},
	Graduate Studies in Mathematics, Vol. 92,
	American Mathematical Society, Providence, RI, 2008.
	
	\bibitem{JL}
	G. James and M. Liebeck,
	\emph{Representations and Characters of Groups},
	2nd ed.,
	Cambridge University Press, Cambridge, 2001.
	
	\bibitem{KS}
	W. Klotz and T. Sander,
	Integral Cayley graphs over abelian groups,
	Electron. J. Combin. \textbf{17} (2010), 1--17.
	
	\bibitem{LZ}
	X. Liu and S. Zhou, Eigenvalues of Cayley graphs, Electron. J. Comb. \textbf{29} (2022), research
	paper p2.9, 164 pp.
	
	\bibitem{LM}
	L. Lu and K. M\"{o}nius,
	Algebraic degree of Cayley graphs over abelian groups and dihedral groups,
	J. Algebr. Comb. \textbf{57} (2023), 585--601.
	
	\bibitem{MSS}
	K. Mönius, J. Steuding and P. Stumpf,
	Which graphs have non-integral spectra?,
	Graphs Combin. \textbf{34} (2018), 1507--1518.
	
	\bibitem{Robinson}
	D. J. S. Robinson,
	\emph{A Course in the Theory of Groups},
	2nd ed.,
	Graduate Texts in Mathematics, Vol. 80,
	Springer-Verlag, New York, 1996.
	
	\bibitem{SM}
N. Sripaisan and Y. Meemark, Algebraic degree of spectra of Cayley hypergraphs, Discrete
Appl. Math. \textbf{316} (2022), 87--94.

\bibitem{WAF}
S. Wang, M. Arezoomand and T. Feng, Algebraic degrees of quasi-abelian semi-Cayley digraphs, Discrete Math. \textbf{347} (2024) 114178.

\bibitem{WYF}
Y. Wu, J. Yang and L. Feng, Algebraic degrees of 2-Cayley digraphs over abelian groups,
Ars Math. Contemp. \textbf{24} (2024) \#P2.02.	
\end{thebibliography}
\end{document}